\documentclass{article}
\usepackage{amsmath}
\usepackage{amssymb}
\usepackage{amsthm}
\usepackage{cite}
\usepackage{stmaryrd}
\SetSymbolFont{stmry}{bold}{U}{stmry}{m}{n} 
\usepackage{slashed} 
\usepackage{euscript}
\usepackage[arrow,curve,matrix,arc,2cell]{xy}
\usepackage[utf8]{inputenc}
\usepackage[unicode]{hyperref}
\usepackage{mathtools}
\usepackage{enumitem}
\DeclareMathOperator\Or{Or}
\DeclareMathOperator\pr{pr}
\begin{document}
\def\e#1\e{\begin{equation}#1\end{equation}}
\def\ea#1\ea{\begin{align}#1\end{align}}
\def\eq#1{{\rm(\ref{#1})}}
\theoremstyle{plain}
\newtheorem{thm}{Theorem}[section]
\newtheorem{lem}[thm]{Lemma}
\newtheorem{prop}[thm]{Proposition}
\newtheorem{cor}[thm]{Corollary}
\newtheorem{quest}[thm]{Question}
\theoremstyle{definition}
\newtheorem{dfn}[thm]{Definition}
\newtheorem{con}[thm]{Construction}
\newtheorem{ex}[thm]{Example}
\newtheorem{rem}[thm]{Remark}
\numberwithin{equation}{section}
\def\dim{\mathop{\rm dim}\nolimits}
\def\codim{\mathop{\rm codim}\nolimits}
\def\Im{\mathop{\rm Im}\nolimits}
\def\det{\mathop{\rm det}\nolimits}
\def\Ker{\mathop{\rm Ker}}
\def\Coker{\mathop{\rm Coker}}
\def\Flag{\mathop{\rm Flag}\nolimits}
\def\FlagSt{\mathop{\rm FlagSt}\nolimits}
\def\Iso{\mathop{\rm Iso}\nolimits}
\def\Aut{\mathop{\rm Aut}}
\def\End{\mathop{\rm End}\nolimits}
\def\GL{\mathop{\rm GL}}
\def\SL{\mathop{\rm SL}}
\def\SO{\mathop{\rm SO}}
\def\SU{\mathop{\rm SU}}
\def\Sp{\mathop{\rm Sp}}
\def\Spin{\mathop{\rm Spin}\nolimits}
\def\U{{\mathbin{\rm U}}}
\def\vol{\mathop{\rm vol}}
\def\inc{\mathop{\rm inc}}
\def\ind{{\rm ind}}
\def\rk{\mathop{\rm rk}}
\def\rank{\mathop{\rm rank}\nolimits}
\def\Hom{\mathop{\rm Hom}\nolimits}
\def\id{{\mathop{\rm id}\nolimits}}
\def\Id{{\mathop{\rm Id}\nolimits}}
\def\TopSta{{\mathop{\bf TopSta}\nolimits}}
\def\Ad{\mathop{\rm Ad}}
\def\irr{{\rm irr}}
\def\cs{{\rm cs}}
\def\Top{{\mathop{\bf Top}\nolimits}}
\def\ul{\underline}
\def\bs{\mathbf}
\def\ge{\geqslant}
\def\le{\leqslant\nobreak}
\def\boo{{\mathbin{\mathbf 1}}}
\def\O{{\mathcal O}}
\def\bA{{\mathbin{\mathbb A}}}
\def\bG{{\mathbin{\mathbb G}}}
\def\bL{{\mathbin{\mathbb L}}}
\def\P{{\mathbin{\mathbb P}}}
\def\bT{{\mathbin{\mathbb T}}}
\def\H{{\mathbin{\mathbb H}}}
\def\K{{\mathbin{\mathbb K}}}
\def\R{{\mathbin{\mathbb R}}}
\def\Z{{\mathbin{\mathbb Z}}}
\def\Q{{\mathbin{\mathbb Q}}}
\def\N{{\mathbin{\mathbb N}}}
\def\C{{\mathbin{\mathbb C}}}
\def\CP{{\mathbin{\mathbb{CP}}}}
\def\A{{\mathbin{\mathcal A}}}
\def\G{{\mathbin{\mathcal G}}}
\def\M{{\mathbin{\mathcal M}}}
\def\cB{{\mathbin{\mathcal B}}}
\def\cC{{\mathbin{\mathcal C}}}
\def\cD{{\mathbin{\mathcal D}}}
\def\cE{{\mathbin{\mathcal E}}}
\def\cF{{\mathbin{\mathcal F}}}
\def\cG{{\mathbin{\mathcal G}}}
\def\cH{{\mathbin{\mathcal H}}}
\def\E{{\mathbin{\mathcal E}}}
\def\F{{\mathbin{\mathcal F}}}
\def\cG{{\mathbin{\mathcal G}}}
\def\cH{{\mathbin{\mathcal H}}}
\def\cI{{\mathbin{\mathcal I}}}
\def\cJ{{\mathbin{\mathcal J}}}
\def\cK{{\mathbin{\mathcal K}}}
\def\cL{{\mathbin{\mathcal L}}}
\def\cN{{\mathbin{\mathcal N}\kern .04em}}
\def\cP{{\mathbin{\mathcal P}}}
\def\cQ{{\mathbin{\mathcal Q}}}
\def\cR{{\mathbin{\mathcal R}}}
\def\cS{{\mathbin{\mathcal S}}}
\def\T{{{\mathcal T}\kern .04em}}
\def\cW{{\mathbin{\mathcal W}}}
\def\cX{{\mathcal X}}
\def\cY{{\mathcal Y}}
\def\cZ{{\mathcal Z}}
\def\cV{{\mathcal V}}
\def\cW{{\mathcal W}}
\def\g{{\mathfrak g}}
\def\al{\alpha}
\def\be{\beta}
\def\ga{\gamma}
\def\de{\delta}
\def\io{\iota}
\def\ep{\epsilon}
\def\la{\lambda}
\def\ka{\kappa}
\def\th{\theta}
\def\ze{\zeta}
\def\up{\upsilon}
\def\vp{\varphi}
\def\si{\sigma}
\def\om{\omega}
\def\De{\Delta}
\def\La{\Lambda}
\def\Om{\Omega}
\def\Ga{\Gamma}
\def\Si{\Sigma}
\def\Th{\Theta}
\def\Up{\Upsilon}
\def\Chi{{\rm X}}
\def\Tau{{T}}
\def\Nu{{\rm N}}
\def\pd{\partial}
\def\ts{\textstyle}
\def\st{\scriptstyle}
\def\sst{\scriptscriptstyle}
\def\w{\wedge}
\def\sm{\setminus}
\def\lt{\ltimes}
\def\bu{\bullet}
\def\sh{\sharp}
\def\di{\diamond}
\def\he{\heartsuit}
\def\op{\oplus}
\def\od{\odot}
\def\op{\oplus}
\def\ot{\otimes}
\def\bt{\boxtimes}
\def\ov{\overline}
\def\bigop{\bigoplus}
\def\bigot{\bigotimes}
\def\iy{\infty}
\def\es{\emptyset}
\def\ra{\rightarrow}
\def\rra{\rightrightarrows}
\def\Ra{\Rightarrow}
\def\Longra{\Longrightarrow}
\def\ab{\allowbreak}
\def\longra{\longrightarrow}
\def\hookra{\hookrightarrow}
\def\dashra{\dashrightarrow}
\def\lb{\llbracket}
\def\rb{\rrbracket}
\def\ha{{\ts\frac{1}{2}}}
\def\t{\times}
\def\ci{\circ}
\def\ti{\tilde}
\def\d{{\rm d}}
\def\md#1{\vert #1 \vert}
\def\ms#1{\vert #1 \vert^2}
\def\bmd#1{\big\vert #1 \big\vert}
\def\bms#1{\big\vert #1 \big\vert^2}
\def\an#1{\langle #1 \rangle}
\def\ban#1{\bigl\langle #1 \bigr\rangle}
\def\o{\operatorname{o}}
\def\D{\mathbin{\slashed{\operatorname{D}}}}
\def\sS{\slashed{\mathrm{S}}}
\def\O{\mathcal{O}}
\def\Lambdatop{\Lambda^\mathrm{top}}
\title{Canonical orientations for moduli spaces of $G_2$-instantons with gauge group $\SU(m)$ or $\U(m)$}
\author{Dominic Joyce and Markus Upmeier}
\date{}
\maketitle

\begin{abstract} Suppose
$(X,g)$ is a compact,
spin Riemannian 7-manifold,
with Dirac operator
$\D^g:C^\iy(X,\sS)\ra C^\iy(X,\sS).$
Let $G$ be $\SU(m)$ or $\U(m),$ and $E\ra X$
be a rank $m$ complex bundle with $G$-structure.
Write $\cB_E$ for the infinite-dimensional moduli
space of connections on $E,$ modulo gauge.
There is a natural principal $\Z_2$-bundle
$O^{\D^{\smash{g}}}_E\ra\cB_E$ parametrizing
orientations of $\det\D^g_{\Ad A}$
for twisted elliptic operators
$\D^g_{\Ad A}$ at each $[A]$ in $\cB_E.$
A theorem of Walpuski \cite{Walp2} shows $O^{\D^{\smash{g}}}_E$ is trivializable.

We prove that if we choose an orientation for $\det\D^g,$ and a flag structure on $X$ in the sense of \cite{Joyc3}, then we can define canonical trivializations of $O^{\D^{\smash{g}}}_E$ for all such bundles $E\ra X,$ satisfying natural compatibilities.

Now let $(X,\vp,g)$ be a compact $G_2$-manifold, with $\d(*\vp)=0.$ Then we can consider moduli spaces $\M_E^{G_2}$ of $G_2$-instantons on $E\ra X,$ which are smooth manifolds under suitable transversality conditions, and derived manifolds in general, with $\M_E^{G_2}\subset\cB_E.$ The restriction of $O^{\D^{\smash{g}}}_E$ to $\M_E^{G_2}$ is the $\Z_2$-bundle of orientations on $\M_E^{G_2}.$ Thus, our theorem induces canonical orientations on all such $G_2$-instanton moduli spaces~$\M_E^{G_2}.$

This contributes to the Donaldson--Segal programme \cite{DoSe}, which proposes defining enumerative invariants of $G_2$-manifolds $(X,\vp,g)$ by counting moduli spaces $\M_E^{G_2},$ with signs depending on a choice of orientation.
\end{abstract}

\setcounter{tocdepth}{1}
\tableofcontents

\section{Introduction}
\label{section.1}

This is the third of six papers: Upmeier \cite{Upme}, Joyce--Tanaka--Upmeier \cite{JTU}, this paper, Cao--Gross--Joyce \cite{CGJ}, and the authors \cite{JoUp1,JoUp2}, on orientability, canonical orientations, and spin structures, for gauge-theoretic moduli spaces. 

The first \cite{Upme} proves the Excision Theorem (see Theorem \ref{thm.2.15} below), which relates orientations on different moduli spaces. The second \cite{JTU} develops the general theory of orientations of moduli spaces, and applies it in dimensions 3,4,5 and 6. This paper studies orientations of moduli spaces in dimension 7. It uses results from \cite{JTU,Upme}, but is self-contained and can be read independently. The sequel \cite{CGJ} concerns dimension~8.

Let $X$ be a compact connected manifold, $G$ be $\SU(m)$ or $\U(m)$ for $m\ge 1,$ and $\g$ be the Lie algebra of $G,$ and $E\to X$ be a rank $m$ complex vector bundle with a $G$-structure, so that $E$ is associated to a principal $G$-bundle $Q \to X$ in the vector representation. Let $\Ad E$ be the associated bundle of Lie algebras, the bundle of skew-Hermitian endomorphisms of $E,$ trace-free if~$G=\SU(m).$

\begin{dfn} $\A_E\subset \Omega^1(Q,\g)$ is the \emph{space of connections} on $Q,$ equipped with its affine Fr\'echet structure modelled on $\Omega^1(X,\Ad E).$ The \emph{gauge group} $\cG=\Aut(Q)$ acts continuously on $\A_E$ by pullback. The quotient space $\cB_E \coloneqq \A_E / \cG$ is the \emph{moduli space of connections} on $E,$ as a topological space with the quotient topology. As in \cite[p.~133]{DoKr}, a connection $\nabla \in \A_E$ is \emph{irreducible} if the stabilizer group of $\nabla$ under the $\G$-action on $\A_E$ equals the centre $Z(G).$ Write $\A_E^\irr\subset\A_E$ for the subset of irreducible connections, and $\cB_E^\irr =\A_E^\irr/\G\subset\cB_E$ for the moduli space of irreducible connections.
\end{dfn}

Suppose now that $(X,g)$ is an odd-dimensional compact Riemannian spin manifold with real spinor bundle $\sS\ra X.$ The real Dirac operator coupled to the induced connections on $\Ad E$
defines a family of self-adjoint elliptic operators
\begin{equation}
\label{equation.1.1}
\D^g_{\Ad A} \colon C^\infty(X, \sS\otimes_\R \Ad E) \longra C^\infty(X, \sS\otimes_\R \Ad E),
\qquad \forall A \in \A_E.
\end{equation}

Let $\det \D^g_{\Ad E}$ be the determinant line bundle of this family, a real line bundle over $\A_E,$ and let ${\bar O^{\D^{\smash{g}}}_E\coloneqq \bigl(\det \D^g_{\Ad E}\setminus\{\text{zero section}\}\bigr)\big/{\R_{>0}}}$ be the associated orientation double cover, a principal $\Z_2$-bundle $\bar O^{\D^{\smash{g}}}_E\ra \A_E,$ where $\Z_2=\{\pm 1\}.$ As $\A_E$ is contractible, $\bar O^{\D^{\smash{g}}}_E$ is trivializable, and we have two possible orientations.

For $X$ a compact spin $7$-manifold and $G=\SU(m)$ the argument of Walpuski in \cite[Prop.~6.3]{Walp2} shows that the gauge group acts trivially on the set of trivializations of $\bar O^{\D^{\smash{g}}}_E,$ and \cite[Ex.~2.13]{JTU} implies that this also holds for $G=\U(m).$ Hence $\bar O^{\D^{\smash{g}}}_E$ descends to a principal $\Z_2$-bundle $O^{\D^{\smash{g}}}_E\ra\cB_E,$ and orientations may be constructed equivalently over $\A_E$ or $\cB_E.$ See \cite{JTU} for more details. 

We define a $G_2$-{\it manifold\/} $(X,\vp,g)$ to be a 7-manifold $X$ with a $G_2$-structure $(\vp,g),$ not necessarily torsion-free. (This differs from \cite[\S 10--\S 12]{Joyc1}, where $(\vp,g)$ was supposed torsion-free.) Suppose $(X,\vp,g)$ is a compact $G_2$-manifold with $\d(*\vp)=0.$ As in Donaldson--Thomas \cite{DoTh} and Donaldson--Segal \cite{DoSe}, a connection $A$ on $E$ is called a $G_2$-{\it instanton\/} if its curvature $F_A$ satisfies
\begin{equation*}
F_A\w *\vp=0. 
\end{equation*}
As $\d(*\vp)=0$ the deformation theory of $G_2$-instantons is elliptic, and therefore the moduli space $\M_E^{G_2}$ of irreducible $G_2$-instantons on $E$ modulo gauge is a smooth manifold (of dimension 0) under suitable transversality assumptions, and a derived manifold (of virtual dimension 0) in the sense of \cite{Joyc2,Joyc4,Joyc5,Joyc6} in the general case. Examples and constructions of $G_2$-instantons on compact $G_2$-manifolds are given in~\cite{MNS,SaEa,SaWa,Walp1,Walp2,Walp3}. 

As in \cite[\S 4.1]{JTU}, the restriction of 
$O^{\D^{\smash{g}}}_E\ra\cB_E$ to $\M_E^{G_2}\subset\cB_E$ is the principal $\Z_2$-bundle of orientations of $\M_E^{G_2},$ as a (derived) manifold. Thus $\M_E^{G_2}$ is orientable, and an orientation of $O^{\D^{\smash{g}}}_E\ra\cB_E$ determines an orientation of $\M_E^{G_2}.$ Such orientations are important for the programme of~\cite{DoTh,DoSe}.
\medskip

In the present paper, we solve the problem of defining canonical orientations for $\M_E^{G_2}.$ As for moduli spaces of anti-self-dual instantons in dimension four, where orientations depend on an orientation of $H^0(X)\op  H^1(X)\op H^2_+(X)$ (see Donaldson \cite{Dona} and Donaldson--Kronheimer \cite[Prop.~7.1.39]{DoKr}), this will depend on some additional algebro-topological data, a so-called {\it flag structure\/}~\cite[\S 3.1]{Joyc3}.

Oversimplifying a bit, a flag structure $F$ on a 7-manifold $X$ assigns $F(Y,s)\ab=\pm 1$ to each compact 3-submanifold $Y\subset X$ with a nonvanishing section $s$ of the normal bundle $N_Y\ra Y,$ such that if $s,s'$ have winding number $d(s,s')\in\Z$ then $F(Y,s')=(-1)^{d(s,s')}F(Y,s),$ and if $C\subset X\t[0,1]$ is an compact 4-submanifold with nonvanishing normal section $t$ and boundary $\pd(C,t)=(Y_0\t\{0\},s_0)\amalg(Y_1\t\{1\},s_1)$ then $F(Y_0,s_0)=F(Y_1,s_1).$ See \S\ref{section.3} for more details.

When $(X,\vp,g)$ is a compact $G_2$-manifold one can define an interesting class of minimal $3$-submanifolds $Y$ in $X$ called {\it associative\/ $3$-folds\/} \cite[\S 10.8]{Joyc1}. Compact associative 3-folds have elliptic deformation theory, and form well-behaved moduli spaces $\M^{\rm ass},$ as (derived) manifolds. In the spirit of \cite{DoTh,DoSe}, the first author \cite{Joyc3} discussed defining enumerative invariants of $(X,\vp,g)$ counting associative 3-folds. To determine signs, he defined canonical orientations on moduli spaces
$\M^{\rm ass},$ using the new idea of flag structures. 

Now Donaldson and Segal \cite{DoSe} (see also Walpuski \cite{Walp3}) explain that associative $3$-folds are connected to $G_2$-instantons, as a sequence of $G_2$-instantons $(E,A_i)_{i=1}^\iy$ can `bubble' along an associative $3$-fold $Y$ as $i\ra\iy.$ So the problems of defining canonical orientations on moduli spaces of associative $3$-folds and of $G_2$-instantons should be related. In \cite[Conj.~8.3]{Joyc3}, the first author conjectured that one should define canonical orientations for moduli spaces of $G_2$-instantons using flag structures. This paper proves that conjecture.

We make heavy use of ideas and results from the previous paper \cite{JTU}, recalled in Section \ref{section.3}. Given the $O^{\D^{\smash{g}}}_E\ra\cB_E$ are orientable, \cite[Th.~2.27]{JTU} gives a way to choose orientations on all $O^{\D^{\smash{g}}}_E$ and $\M_E^{G_2}$ after making finitely many {\it algebraic\/} choices. But here we do something different: we construct orientations on all $O^{\D^{\smash{g}}}_E$ and $\M_E^{G_2}$ depending on a {\it geometric structure\/} on $X,$ the flag structure. We use a general procedure for doing this using excision outlined in \cite[\S 3.3]{JTU}. 
\medskip

In \eqref{equation.2.2} we define the {\it orientation\/ $\Z_2$-torsor\/} $\Or_E$
of a $\SU(m)$-bundle $E.$ Up to an orientation for the untwisted Diracian, this is the set of orientations on the determinant line bundle of \eqref{equation.1.1}. For a $\SU(m_1)$-bundle $E_1 \to X$ and $\SU(m_2)$-bundle $E_2\to X$ we have canonical isomorphisms (Proposition~\ref{prop.2.14})
\begin{align}
\Or_{E_1\oplus E_2} &\cong \Or_{E_1} \otimes_{\Z_2} \Or_{E_2}, 
\label{equation.1.2}\\
\Or_{\underline\C^m} &\cong \Z_2, 
\label{equation.1.3}
\end{align}
where \eqref{equation.1.3} corresponds to the `standard orientations' of~\cite[\S 2.2.2]{JTU}.

Here is our main result. The proof is sketched below.

\begin{thm}
\label{thm.1.2}
A flag structure\/ $F$ on a compact spin\/ $7$-manifold\/ $X$ determines, for every\/ $\SU(m)$-bundle\/ $E \to X$ and\/ $m\in \N,$ a canonical orientation
\e
\label{equation.1.4}
o^F(E) \in \Or_E
\e
satisfying the following axioms, by which\/ $o^F(E)$ is uniquely determined:
\begin{itemize}
\setlength{\itemsep}{0pt}
\setlength{\parsep}{0pt}
\item[{\bf(a)}]\textup{(Normalization.)}~For $E=\underline\C^m$ trivial, let\/ $o^\mathrm{flat}(E) \in \Or_E$
be the image of\/ $1 \in \Z_2$ under the isomorphism \eqref{equation.1.3}. Then
\e
\label{equation.1.5}
o^{F}(E) = o^\mathrm{flat}(E).
\e
\item[{\bf(b)}]\textup{(Stabilization.)}~Under the isomorphism\/ $\Or_{E\oplus \underline\C} \cong \Or_E \otimes_{\Z_2} \Or_{\underline\C} \cong \Or_E,$ using \eqref{equation.1.2} and\/ {\rm\eqref{equation.1.3},} we have
\e
\label{equation.1.6}
o^{F}(E\oplus \underline\C) = o^{F}(E).
\e
\item[{\bf(c)}]\textup{(Excision.)}~Let\/ $E^\pm \to X^\pm$ be\/ $\SU(m)$-bundles over a pair of compact spin\/ $7$\nobreakdash-manifolds with flag structures\/ $F^\pm.$ Let\/ $\rho^\pm$ be\/ $\SU(m)$-frames of\/ $E^\pm$ outside compact subsets of open\/ $U^\pm \subset X^\pm.$ Let\/ ${\Phi\colon E^+|_{U^+} \to \phi^*(E^-|_{U^-})}$ be a\/ $\SU(m)$\nobreakdash-isomorphism covering a spin diffeomorphism\/ ${\phi \colon U^+ \to U^-}.$ Assume\/ $\Phi\circ \rho^+ = \phi^*\rho^-$ outside a compact subset of\/ $U^+.$ Under the excision isomorphism of Theorem~\textup{\ref{thm.2.15}} we then have
\e
\label{equation.1.7}
\!\!\!\!\!\!\!\!\!\!\!\Or(\phi,\Phi,\rho^+,\rho^-) \bigl( o^{F^+}(E^+) \bigr) \!=\! \bigl(F^+\vert_{U^+}/\phi^*(F^-\vert_{U^-})\bigr)(\al^+)\cdot o^{F^-}(E^-),
\e
where\/ $\al^+ \in H_3(U^+;\Z)$ is the homology class Poincar\'e dual to the relative Chern class\/ $c_2(E^+|_{U^+},\rho^+) \in H^4_\mathrm{cpt}(U^+;\Z).$
\end{itemize}
Moreover, the following additional properties hold:
\begin{itemize}
\setlength{\itemsep}{0pt}
\setlength{\parsep}{0pt}
\item[{\bf(i)}]\textup{(Direct sums.)}~Let\/ $E_1 \to X$ be a\/ $\SU(m_1)$-bundle and\/ $E_2 \to X$ a\/ $\SU(m_2)$-bundle.
Under the isomorphism \eqref{equation.1.2} we then have
\e
\label{equation.1.8}
o^F(E_1\oplus E_2) = o^F(E_1) \otimes o^F(E_2).
\e
\item[{\bf(ii)}]\textup{(Families.)}~Let\/ $P$ be a compact Hausdorff space,\/ $X$ a compact spin\/ $7$-manifold, and\/ $E \to X \times P$ a\/ $\SU(m)$-bundle. The union of all torsors\/ $\Or(E|_{X\times \{p\}})$ for each\/ $p\in P$ is
a double cover of\/ $P,$ of which the map\/ ${p\mapsto o^F(E|_{X\times \{p\}})}$ defines a continuous section. In particular, canonical orientations are deformation invariant.
\end{itemize}

Now let\/ $E\ra X$ be a rank\/ $m$ complex vector bundle with\/ $\U(m)$-structure. Then\/ $\ti E=E\op\La^mE^*$ is a rank\/ $m+1$ complex vector bundle with\/ $\SU(m+1)$-structure, and\/ {\rm\cite[Ex.~2.13]{JTU}} defines a canonical isomorphism of\/ $\Z_2$-torsors\/ $\Or_E\cong\Or_{\smash{\ti E}}.$ Hence the first part gives canonical orientations\/ $o^F(E) \in \Or_E$ for all\/ $\U(m)$-bundles\/ $E\ra X.$ These satisfy the analogues of\/ {\bf(a)\rm--\bf(c)} and\/ {\bf\bf(ii)\rm,} but may not satisfy\/~{\bf(i)}.
\end{thm}

\begin{rem} The problem with extending (i) to $\U(m)$-bundles in the last part, is that if $E_1,E_2\ra X$ are $\U(m_1)$- and $\U(m_2)$-bundles then the left hand side of \eqref{equation.1.8} comes from the orientation for the $\SU(m_1+m_2+1)$-bundle $(E_1\op E_2)\op\La^{m_1+m_2}(E_1\op E_2)^*,$ but the right hand side comes from the orientation for the $\SU(m_1+m_2+2)$-bundle $(E_1\op\La^{m_1}E_1^*)\op(E_2\op\La^{m_2}E_2^*),$ which is different.

The orientations $o^F(E)$ for $\U(m)$-bundles defined in the last part may not satisfy (i). For example, let $X=\CP^3\t\cS^1,$ which has two flag structures $F^+,F^-,$ and take $E_1=\pi_{\CP^3}^*(\O(k))$ and $E_2=\pi_{\CP^3}^*(\O(l))$ for $k,l\in\Z$ odd. Using \eqref{equation.1.7} we find that changing from $F^+$ to $F^-$ changes the sign of all three of $o^{F^\pm}(E_1),o^{F^\pm}(E_2),o^{F^\pm}(E_1\oplus E_2),$ so \eqref{equation.1.8} holds for only one of~$F^+,F^-.$
 
It may still be possible to choose orientations $o^F(E)$ for all $\U(m)$-bundles $E\ra X$ satisfying (a),(b),(i),(ii), and perhaps (c), by a different method.
\end{rem}

One application of this theorem is to the problem of defining orientations for moduli spaces of $G_2$-instantons $\M_E^{G_2}.$ As the moduli space is zero-dimensional, there are many arbitrary orientations, so the point of the problem is to come up with a natural assignment, in particular one that is stable under deformations of the $G_2$-structure. Following \cite[\S 4.1]{JTU}, we have already explained how Walpuski \cite[Prop.~6.3]{Walp2} and Theorem~\ref{thm.1.2} imply the following:

\begin{cor}
\label{cor.1.4}
Let\/ $(X,\vp,g)$ be a compact\/ $G_2$-manifold with\/ $\d(*\vp)=0,$ and choose an orientation of\/ $\det \D^g$ for the untwisted Diracian and a flag structure\/ $F$ on $X.$ Then we can define a canonical orientation for the moduli space\/ $\M_E^{G_2}$ of\/ $G_2$-instantons on\/ $X$ whenever\/ $E\ra X$ is a\/ $\SU(m)$- or\/ $\U(m)$-bundle.
\end{cor}

Donaldson and Segal \cite{DoSe} propose defining enumerative invariants of $(X,\vp,g)$ by counting $\M_E^{G_2},$ {\it with signs}, and adding correction terms from associative 3-folds in $X.$ To determine the signs we need an orientation of $\M_E^{G_2}.$ Thus, Corollary \ref{cor.1.4} contributes to the Donaldson--Segal programme.

It is natural to want to extend Theorem \ref{thm.1.2} and Corollary \ref{cor.1.4} to moduli spaces of connections on principal $G$-bundles $Q\ra X$ for Lie groups $G$ other than $\SU(m)$ and $\U(m),$ but this is not always possible. Section \ref{subsection.2.4} gives an example of a compact, spin 7-manifold $X$ for which $O^{\D^{\smash{g}}}_Q\ra\cB_Q$ is not orientable when $Q=X\t\Sp(m)\ra X$ is the trivial $\Sp(m)$-bundle, for all~$m\ge 2.$

In the sequels \cite{JoUp1,JoUp2} we use Theorem \ref{thm.1.2} to construct `spin structures' on moduli spaces $\cB_P$ for principal $\U(m)$- or $\SU(m)$-bundles $P\ra X$ over a compact spin 6-manifold $X,$ and apply this to construct `orientation data' for Calabi--Yau 3-folds $X,$ as in Kontsevich and Soibelman \cite[\S 5]{KoSo}, solving a long-standing problem in Donaldson--Thomas theory.

\subsubsection*{Outline of the paper}

We begin in \S\ref{section.2} by recalling background material on determinant line bundles. Then our main object of study, the orientation torsor $\Or_E$ of a $\SU(m)$-bundle $E\to X,$ is introduced along with its basic properties. We recall from \cite{Upme} the excision technique from index theory in the context of orientations. It can be regarded as extending the functoriality of orientation torsors from globally defined isomorphisms to local ones. Section \ref{section.3} briefly recalls flag structures, and \S\ref{section.4} proves Theorem \ref{thm.1.2}. In brief, the proof works as follows:
\begin{itemize}
\setlength{\itemsep}{0pt}
\setlength{\parsep}{0pt}
\item[(A)] Let $X$ be a compact spin 7-manifold with flag structure $F,$ and $E\ra X$ a $\SU(m)$-bundle. We show that we can find:
\begin{itemize}
\setlength{\itemsep}{0pt}
\setlength{\parsep}{0pt}
\item[(a)] A compact 3-submanifold $Y\subset X.$
\item[(b)] An $\SU(m)$-trivialization $\rho:\ul{\C}^m\vert_{X\sm Y}\,\smash{{\buildrel\cong\over\longra}}\,\ab E\vert_{X\sm Y}.$
\item[(c)] An embedding $\io:Y\hookra\cS^7,$ so $Y'=\io(Y)$ is a 3-submanifold of $\cS^7.$
\item[(d)] An isomorphism $\Psi:N_Y\ra \io^*(N_{Y'})$ between the normal bundles of $Y$ in $X$ and $Y'$ in $\cS^7,$ preserving orientations and spin structures.
\item[(e)] Tubular neighbourhoods $U$ of $Y$ in $X$ and $U'$ of $Y'$ in $\cS^7,$ and a spin diffeomorphism $\psi:U\ra U'$ with $\psi\vert_Y=\io$ and~$\d\psi\vert_{N_Y}=\Psi.$
\end{itemize}
Define a $\SU(m)$-bundle $E'\ra\cS^7$ by $E'\vert_{\cS^7\sm Y'}\cong\ul{\C}^m,$ $E'\vert_{U'}\cong\psi_*(E\vert_U),$ identified over $U'\sm Y'$ by $(\psi\vert_{U\sm Y})_*(\rho),$ with $\Xi:E\vert_U\,{\buildrel\cong\over\longra}\,\psi^*(E'\vert_{U'}).$ Then we have an excision isomorphism~$\Or(\psi,\Xi,\rho,\rho')\colon \Or_E\ra \Or_{E'}.$

Now every $\SU(m)$-bundle on $\cS^7$ is stably trivial, so Theorem \ref{thm.1.2}(a),(b) determine a unique orientation $o^\mathrm{flat}(E')\in \Or_{E'}.$ Following Theorem \ref{thm.1.2}(c) we define an orientation $o^F_{Y,\rho,\io,\Psi}(E)\in\Or_E$ by
\begin{equation}
\label{equation.1.9}
o^F_{Y,\rho,\io,\Psi}(E)=\bigl(F\vert_U/\psi^*(F_{\cS^7}\vert_{U'})\bigr)[Y]
\cdot\Or(\psi,\Xi,\rho,\rho')^{-1}(o^\mathrm{flat}(E')),
\end{equation}
where $F_{\cS^7}$ is the unique flag structure on $\cS^7.$

Observe that if Theorem \ref{thm.1.2}(a)--(c) hold, they force $o^F(E)=o^F_{Y,\rho,\io,\Psi}(E).$ Thus, if orientations $o^F(E)$ exist satisfying Theorem \ref{thm.1.2}(a)--(c), then they are uniquely determined, as claimed.
\item[(B)] We prove that $o^F_{Y,\rho,\io,\Psi}(E)$ is independent of the choices in~(A)(a)--(e): 
\begin{itemize}
\setlength{\itemsep}{0pt}
\setlength{\parsep}{0pt}
\item[(i)] Independence of $U,U',\psi$ for fixed $Y,\rho,\io,\Psi$ is obvious from properties of excision isomorphisms.
\item[(ii)] Independence of $\Psi$ for fixed $Y,\rho,\io$ is nontrivial. Given two different choices $\Psi_0,\Psi_1$ and $\psi_0,\psi_1,$ we compute the signs comparing how $\psi_0,\psi_1$ act on orientations of bundles trivial away from $Y,$ and how $\psi_0,\psi_1$ act on flag structures near $Y,$ and show these signs are the same, so the combined effect of both signs in \eq{equation.1.9} cancels out.

This is the main point where flag structures are used in the proof.
\item[(iii)] Independence of $\io:Y\hookra\cS^7$ for fixed $Y,\rho$ is easy, as any two such embeddings are isotopic through embeddings.
\item[(iv)] Independence of $Y,\rho$ is again nontrivial, and is proved by analyzing a bordism $Z\subset X\t[0,1]$ between two choices~$Y_0,Y_1\subset X.$
\end{itemize}
We can now define $o^F(E)=o^F_{Y,\rho,\io,\Psi}(E)$ for all $X,F$ and~$E\ra X.$ 
\item[(C)] We verify the $o^F(E)$ in (B) satisfy Theorem \ref{thm.1.2}(a)--(c),(i)--(ii).
\item[(D)] We extend from $\SU(m)$-bundles to $\U(m)$-bundles, which is easy.
\end{itemize}
\smallskip

\noindent{\it Acknowledgements.} This research was partly funded by a Simons Collaboration Grant on `Special Holonomy in Geometry, Analysis and Physics'. The second author was funded by DFG grant UP~85/3-1 and by grant UP~85/2-1 of the DFG priority program SPP~2026 `Geometry at Infinity.' The authors would like to thank Yalong Cao, Aleksander Doan, Sebastian Goette, Jacob Gross, Andriy Haydys, Johannes Nordstr\"om, Yuuji Tanaka, Richard Thomas and Thomas Walpuski for helpful conversations, and the referee. 

\section{Orientations and determinants}
\label{section.2}

\subsection{The Quillen determinant}

\subsubsection{Finite dimensions}

For finite-dimensional vector spaces, the top exterior power has the fundamental property
that a short exact sequence
\begin{equation*}
\smash{\xymatrix{
0\ar[r]&
U\ar[r]^f&
V\ar[r]^g&
W\ar[r]&
0 }}
\end{equation*}
induces a \emph{canonical} isomorphism
\begin{equation*}
  \Lambdatop U \otimes \Lambdatop W \cong \Lambdatop V.
\end{equation*}

\begin{lem}
For finite-dimensional vector spaces $V$ and $W$ we have
\begin{equation*}
\Lambdatop (V\otimes W) \cong (\Lambdatop V)^{\otimes \dim W} \otimes (\Lambdatop W)^{\otimes \dim V}.
\end{equation*}
\end{lem}

\subsubsection{Fredholm determinant}

The determinant of a homomorphism
$f\colon V^0 \to V^1$ of finite-dimensional vector spaces is an element of $(\Lambdatop V^0)^* \otimes \Lambdatop V^1.$ This is isomorphic to $(\Lambdatop \Ker f)^* \otimes \Lambdatop \Coker f,$
by the fundamental property applied to
\begin{equation*}
\xymatrix@C=3ex{
0\ar[r]&
\Ker f\ar[r]&
V^0\ar[r]&
V^1\ar[r]&
\Coker f\ar[r]&
0.
}
\end{equation*}

\begin{dfn}
Let $F\colon \cH^0 \to \cH^1$ be a Fredholm operator between Hilbert spaces. 
The \emph{determinant line of $F$} is $\det F \coloneqq \Lambdatop \Ker F \otimes \left(\Lambdatop \Coker F \right)^*.$
\end{dfn}

\begin{prop}\label{prop.2.3}
For every commutative diagram of bounded operators
\begin{equation*}
\xymatrix@C=40pt@R=15pt{
0\ar[r] & \cF^0 \ar[r]\ar[d]_{F} & \cG^0\ar[d]_{G}\ar[r] & \cH^0\ar[d]_{H}\ar[r] & 0\\
0\ar[r] & \cF^1 \ar[r] & \cG^1\ar[r] & \cH^1\ar[r] & 0 }
\end{equation*}
with exact rows and\/ $F,G,H$ Fredholm there is a canonical isomorphism
\e
\label{equation.2.1}
\det G \cong \det F\otimes \det H.
\e
\end{prop}

\begin{proof}
Snake lemma and the fundamental property in finite dimensions.
\end{proof}

\begin{dfn}
Let $T$ be a paracompact Hausdorff space.
A \emph{$T$-family} of Fredholm operators
$\{F_t \colon \cH^0_t \to \cH^1_t\}_{t\in T}$
is a homomorphism $F \colon \cH^0 \to \cH^1$
of Hilbert space bundles over $T$ whose
restriction to every fibre is Fredholm. 
The \emph{determinant line bundle} of $F$ is $\det F \coloneqq \bigsqcup_{t\in T} \det F_t.$
\end{dfn}

To see that $\det F$ is locally trivial, pick $t_0 \in T$ and $s(t_0)\colon \C^k \to \cH^1_{t_0}$ 
so that $F_{t_0}\oplus s(t_0)$ is surjective. Extend $s$ to a neighbourhood of $t_0.$ Proposition~\ref{prop.2.3} for $(F,F\oplus s,\underline\C^k \to \{0\})$ gives ${\det F = \det(F\oplus s) = \Lambdatop \Ker(F\oplus s)^*}.$ Since $F\oplus s$ is surjective near $t_0,$ $\Ker(F\oplus s)$ is a subbundle there.

\begin{ex}
Let $D$ be a family of elliptic differential or pseudo-differential operators over a compact manifold $X.$ These determine Fredholm operators by regarding them as acting on Sobolev spaces. 
The determinant line bundle is independent of the degree of the Sobolev space, since by elliptic regularity the kernels of $D$ and $D^*$ consist of smooth sections. Here, $D^*$ denotes the formally adjoint differential operator and we recall $\Ker D^* \cong \Coker D.$
\end{ex}

For a family of differential operators the manifold and vector bundle may depend on $t\in T,$ as long as they form a fibre bundle \cite{AtSi4}.

\begin{lem}
\label{lem.2.6}
Let\/ $\{F_t^0 \colon \cH^0_t \to \cH^1_t\}_{t\in T}$ and\/
$\{F_t^1 \colon \cH^0_t \to \cH^1_t\}_{t\in T}$ be homotopic through\/ $T$-families of Fredholm operators. Then\/ $\det F^0 \cong \det F^1.$
\end{lem}

\begin{proof} By definition, a homotopy is a $(T\times [0,1])$-family of Fredholm operators
$H(t,s).$ The determinant line bundle of $H$
restricts over $T\times \{s\}$ to $\det F^s$ for $s=0,1.$ 
The inclusions of the endpoints of $T\times [0,1]$ are
homotopic and as $T$ is paracompact Hausdorff, the pullbacks $\det H|_{T\times \{0\}}$ and
$\det H|_{T\times \{0\}}$ are isomorphic.
\end{proof}

Up to this point the discussion applies to operators over both the real
or the complex numbers. From now on we need real operators.

\begin{dfn} The \emph{orientation cover} of a $T$-family of real Fredholm operators $\{F_t \colon \cH^0_t \to \cH^1_t\}_{t\in T}$ is
$\Or F \coloneqq (\det F \setminus \{\text{zero section}\}) / \R_{>0}.$
An {\it orientation\/} for the determinant of the family is a
global section of $\Or F.$
\end{dfn}

As $\det F$ is locally trivial, $\Or F$ is a double cover of $T,$ so
for $T$ connected there are either two orientations or none.
An advantage of orientation covers is their deformation
invariance. The argument for Lemma~\ref{lem.2.6} now gives:

\begin{lem}
\label{lem.2.8}
Let\/ $\{F_t^0 \colon \cH^0_t \to \cH^1_t\}_{t\in T}$ and\/ $\{F_t^1 \colon \cH^0_t \to \cH^1_t\}_{t\in T}$ be homotopic through\/ $T$-families of real Fredholm operators. Then we have a canonical
fibre transport isomorphism\/ $\Or F^0 \cong \Or F^1.$
\end{lem}

In particular, the orientation cover of a $T$-family of real elliptic operators $D$ depends only on the principal symbol.

\subsection{Orientation torsors and excision}

\subsubsection{Basic construction}

We now simplify the discussion by restricting to Diracians twisted by
connections. On the level of orientations only the underlying vector
bundles matter:

\begin{dfn}
\label{dfn.2.9}
Let $(X,g)$ be an odd-dimensional compact spin manifold with real spinor
bundle $\sS.$ Let $E \to X$ be a vector bundle with $\SU(m)$-structure,
and let $\Ad E$ be the associated bundle of Lie algebras. The twisted Diracians
\begin{equation*}
\D^g_{\Ad A} \colon C^\infty(X, \sS\otimes_\R \Ad E) \longra C^\infty(X, \sS\otimes_\R \Ad E),\qquad
A\in \A_E,
\end{equation*}
determine a family $\D^g_{\Ad E}$ of real elliptic operators parametrized by the space $\A_E$
of $\SU(m)$-connections on $E.$
Let $\D^g_{\Ad \underline\C^m,0}$ be the Diracian twisted by the trivial bundle
$\Ad \underline\C^m$ and zero connection.
The \emph{orientation torsor of\/ $E\to X$} is
\e
\label{equation.2.2}
\Or_E \coloneqq C^\infty\bigl(  \A_E, \bar O^{\D^{\smash{g}}}_E \bigr)\otimes_{\Z_2} \Or\bigl(\det(\D^g_{\Ad \underline\C^m,0})\bigr)^*.
\e
Similarly, for a paracompact Hausdorff space $P$ and a $P$-family of
$\SU(m)$-bundles, meaning a $\SU(m)$-bundle $E \to X \times P$ smooth in the $X$ directions, we get a double cover $\Or_E \to P$ by taking global sections only in the $X$-direction.
\end{dfn}

By Lemma \ref{lem.2.8}, $\Or_E$ does not depend on $g$ up
to canonical isomorphism. More formally, one may take global
sections also in this contractible variable.

\begin{rem} Let $Q$ be the principal $\SU(m)$-frame bundle of $E.$ In the terminology of \cite{JTU}, when $\cB_Q$ is n-orientable, the orientation torsor $\Or_E$ is the set of global sections of the n-orientation bundle~$\check{O}^{\D^g}_Q\ra\cB_Q.$
\end{rem}

\begin{rem} As $\D^g_{\Ad \underline\C^m,0}$ is symmetric, the second factor in \eqref{equation.2.2} is canonically trivial. However, when $m$ is even, this orientation is sensitive to the metric and changes discontinuously according to the
spectral flow of $\D^g.$ We prefer to keep track of an extra choice of orientation for the untwisted Diracian $\D^g.$ By \eqref{equation.2.1} it induces a trivialization of $\Or\bigl(\det(\D^g_{\Ad \underline\C^m,0})\bigr).$ The second factor in \eqref{equation.2.2} has been introduced to simplify the formulation of the excision principle below.
\end{rem}

\begin{rem} For anti-self-dual moduli spaces in dimension four the Diracian is replaced by $\d\op\d^*_+:C^\iy(\La^0T^*X\op\La^2_+T^*X)\ra C^\iy(\La^1T^*X),$ as in Donaldson--Kronheimer \cite{DoKr}. For these $\Or_E$ is canonically trivial and the untwisted operator is responsible for the dependence of orientations on $H^0(X)\oplus H^1(X)\oplus H^2_+(X).$
\end{rem}

\begin{dfn}
\label{dfn.2.13}
For $E=\underline{\C}^m$ we can evaluate at the zero connection and canonically identify the orientation torsor with $\Z_2.$ We write
$o^\mathrm{flat}(\underline{\C}^m) \in \Or_{\underline{\C}^m}$ for
this canonical base-point.
\end{dfn}

\subsubsection{Orientations and direct sums}

The behaviour of orientation bundles under direct sums is studied in 
the companion paper \cite[Ex.~2.11]{JTU}. From there we recall the following:

\begin{prop}
\label{prop.2.14}
Let\/ $E_1$ be a\/ $\SU(m_1)$-bundle,\/ $E_2$ a\/ $\SU(m_2)$-bundle over a compact odd-dimensional spin manifold\/ $X.$ We have a canonical isomorphism
\e
\label{equation.2.3}
\lambda_{E_1, E_2}\colon \Or_{E_1} \otimes_{\Z_2} \Or_{E_2}
\longra
\Or_{E_1\oplus E_2}.
\e
These have the following properties:
\begin{itemize}
\setlength{\itemsep}{0pt}
\setlength{\parsep}{0pt}
\item[{\bf(i)}]\textup{(Families.)}~Let\/ $P$ be compact Hausdorff,\/ $E_1 \to X\times P$ a\/ $\SU(m_1)$-bundle,
and\/ $E_2 \to X\times P$ a\/ $\SU(m_2)$-bundle, regarded as\/ $P$-families of bundles. Then the collection of all maps\/ $\lambda_{E_1|_{X\times \{p\}}, E_2|_{X\times \{p\}}}$ for each\/ $p \in P$ becomes a continuous map of double covers over\/ $P.$
\item[{\bf(ii)}]\textup{(Associative.)}~$\lambda_{E_1, E_2\oplus E_3}\circ
(\id_{\Or_{E_1}}\otimes\lambda_{E_2, E_3})
=\lambda_{E_1\oplus E_2, E_3}\circ
(\lambda_{E_1, E_2} \otimes \id_{\Or_{E_3}})$
\item[{\bf(iii)}]\textup{(Commutative.)}~$\Or(\operatorname{flip})\circ \lambda_{E_1, E_2}\!=\!\lambda_{E_2, E_1}\circ\operatorname{flip} \colon\! \Or_{E_1}\otimes \Or_{E_2}\! \to\! \Or_{E_2 \oplus E_1}.$
\item[{\bf(iv)}]\textup{(Unital.)}~$\lambda_{\underline\C^{m_1},\underline\C^{m_2}}\bigl(o^\mathrm{flat}(\underline\C^{m_1})\otimes o^\mathrm{flat}(\underline\C^{m_2})\bigr) =
o^\mathrm{flat}(\underline\C^{m_1+m_2}).$
\end{itemize}
Moreover, in \eqref{equation.2.5} we will see that the isomorphisms \eqref{equation.2.3} are natural.
\end{prop}

We shall adopt the product notation $u \cdot v\coloneqq \lambda_{E_1, E_2}(u\otimes v).$

\begin{proof} We briefly recall the argument of \cite[Ex.~2.11]{JTU}.
For the adjoint bundles $\Ad(E_1\oplus E_2)\cong \Ad(E_1) \oplus \Ad(E_2)\oplus \R \oplus \Hom_\C(E_1,E_2),$ so by \eqref{equation.2.1}
\begin{equation*}
\bar O^{\D^{\smash{g}}}_{E_1\op E_2}\cong\bar O^{\D^{\smash{g}}}_{E_1}\otimes_{\Z_2}\bar O^{\D^{\smash{g}}}_{E_2}\otimes_{\Z_2} \Or\bigl(\det_\R \D^g\bigr)\otimes_{\Z_2} \Or\bigl(\det_\R(\D^g_{\Hom_\C(E_1,E_2)})\bigr).
\end{equation*}
As the Diracian twisted by $\Hom_\C(E_1,E_2)$ is complex linear, its kernels and cokernels are complex vector spaces and $\Or\bigl(\det_\R(\D^g_{\Hom_\C(E_1,E_2)})\bigr)$ is canonically trivial. This, combined with the same for $\underline\C^{m_1},$ $\underline\C^{m_2}$ in place of $E_1,$ $E_2,$ gives \eqref{equation.2.3}. The same proof works for families. Associativity is \cite[(2.12)]{JTU} and commutativity is \cite[(2.11)]{JTU}, noting that indices vanish in odd dimensions.
\end{proof}

\subsubsection{Excision}

Seeley's excision principle \cite[Th.~1 on p.~198]{Seel} (also called transplanting) is one of the key techniques in the $K$-theory proof of the Atiyah--Singer index theorem \cite[\S 8]{AtSi1}. Donaldson first applied excision to gauge theory in \cite{Dona}, see also \cite[\S 7]{DoKr}. In \cite{Upme}, the second author observes that on the level of orientations these ideas can be formalized into a `categorification' of the classical calculus for the numerical index. Here is \cite[Th.~2.13]{Upme} in the case~$G=\SU(m)$:

\begin{thm}[Excision]\label{thm.2.15}
Let\/ $E^\pm \to X^\pm$ be\/ $\SU(m)$-bundles over compact connected spin manifolds. Let\/ $U^\pm \subset X^\pm$ be open and let\/ $\rho^\pm$ be\/ $\SU(m)$-frames of\/ $E^\pm$ defined outside compact subsets of\/ $U^\pm.$ Let\/ $\phi \colon U^+ \to U^-$ be a spin diffeomorphism covered by a\/ $\SU(m)$-isomorphism\/ $\Phi\colon E^+|_{U^+} \to E^-|_{U^-}$ with\/ $\Phi\circ \rho^+ = \phi^*\rho^-$ outside some compact subset of\/ $U^+.$ This data induces an \emph{excision isomorphism}
\e
\label{equation.2.4}
\Or(\phi,\Phi,\rho^+,\rho^-) \colon \Or_{E^+} \longra \Or_{E^-}.
\e
These excision isomorphisms have the following properties:
\begin{itemize}
\setlength{\itemsep}{0pt}
\setlength{\parsep}{0pt}
\item[{\bf(i)}]\textup{(Functoriality.)}~Let\/ $E^\times \to X^\times$ be a\/ $\SU(m)$-bundle,\/ $U^\times \subset X^\times$ open,\/ $\psi\colon U^- \to U^\times$ a spin diffeomorphism,\/ $\rho^\times$ a\/ $\SU(m)$-frame defined outside a compact subset of\/ $U^\times,$ and\/ $\Psi$ a $\SU(m)$-isomorphism covering\/ $\psi$ that identifies\/ $\rho^-$ and\/ $\rho^\times$ outside a compact subset of\/ $U^-.$ Then
\begin{equation*}
\Or(\psi,\Psi,\rho^-,\rho^\times)\circ \Or(\phi,\Phi,\rho^+,\rho^-) = \Or(\psi\circ\phi,\Psi\circ\Phi,\rho^+,\rho^\times).
\end{equation*}
Moreover,\/ $\Or(\id,\id,\rho^+,\rho^-)=\id_{\Or_E}.$
\item[{\bf(ii)}]\textup{(Families.)}~Let\/ ${E^\pm \to X^\pm \times P}$ be\/ $\SU(m)$\nobreakdash-bundles, where\/ $P$ is a compact Hausdorff space. Let\/ $\rho^\pm$ be\/ $\SU(m)$-frames of\/ $E^\pm$ outside compact subsets of open\/ $U^\pm \subset X^\pm.$ Let\/ $\Phi \colon E^+|_{U^+} \to E^-|_{U^-}$ be a\/ $\SU(m)$-isomorphism covering a continuous\/ $P$-family of spin diffeomorphisms\/ $\phi\colon U^+ \to U^-.$ Assume\/ ${\Phi\circ \rho^+ = \phi^*\rho^-}$ outside a compact subset of\/ $U^+.$ Then the collection of all maps \eqref{equation.2.4} for each\/ $p\in P$ becomes a continuous map of double covers over\/~$P.$

In particular, when\/ $E^\pm$ is pulled back from\/ $X^\pm$ along the projection, the isomorphism \eqref{equation.2.4} is unchanged under deformation of the rest of the data\/ $U^\pm, \rho^\pm, \phi, \Phi.$
\item[{\bf(iii)}]\textup{(Empty set.)}~If\/ $U^\pm = \emptyset$ then\/ $\Or(\phi,\Phi,\rho^+,\rho^-)=\id_{\Z_2}$ under the isomorphisms\/ $\Or_{E^\pm}\cong \Z_2$ induced by Definition\/ {\rm\ref{dfn.2.13}} and\/~$E^+\,{\buildrel\rho^+\over =}\,\ul\C^m\,{\buildrel\rho^-\over =}\,E^-.$
\item[{\bf(iv)}]\textup{(Direct sums.)}~For\/ $k=1,2$ let\/ $E_k^\pm \to X^\pm$ be\/ $\SU(m_k)$-bundles and let\/ $\rho^\pm_k$ be\/ $\SU(m_k)$-frames of\/ $E_k^\pm$ outside compact subsets of\/~$U^\pm \subset X^\pm.$

Let\/ $\phi \colon U^+ \to U^-$ be a spin diffeomorphism covered by $\SU(m_k)$-iso\-mor\-ph\-isms\/ $\Phi_k \colon E_k^+ \to E_k^-$ for\/ $k=1,2.$ Then we have a commutative diagram
\e
\label{equation.2.5}
\begin{gathered}
\xymatrix@C=180pt@R=15pt{
*+[r]{\Or_{E_1^+}\otimes \Or_{E_2^+}} \ar[d]^{\eqref{equation.2.3}}\ar[r]_{\Or(\phi,\Phi_1, \rho_1^\pm)\otimes \Or(\phi,\Phi_2, \rho_2^\pm)} & *+[l]{\Or_{E_1^-}\otimes \Or_{E_2^-}} \ar[d]_{\eqref{equation.2.3}} \\
*+[r]{\Or_{E_1^+ \oplus E_2^+}} \ar[r]^{\Or(\phi, \Phi_1\oplus \Phi_2, \rho_1^\pm \oplus \rho_2^\pm)} & *+[l]{\Or_{E_1^- \oplus E_2^-}.\!} }
\end{gathered}
\e
\item[{\bf(v)}]\textup{(Restriction.)}~Let\/ $\tilde\phi\colon \tilde{U}^+ \to \tilde{U}^-$ be a spin diffeomorphism extending\/ $\phi$ to open supersets\/ $U^\pm \subset \tilde{U}^\pm \subset X^\pm,$ let\/ $\tilde\Phi$ be a\/ $\SU(m)$-isomorphism over\/ $\tilde\phi$ extending\/ $\Phi,$ and assume\/ $\tilde\Phi\circ \rho^+ = \tilde\phi^*\rho^-$ outside a compact subset of\/ $U^+.$ Then\/ $\Or(\phi,\Phi,\rho^+,\rho^-) = \Or(\tilde\phi,\tilde\Phi,\rho^+,\rho^-).$
\end{itemize}
\end{thm}

Here we recall from \cite[p.~86]{LaMi} that a \emph{spin diffeomorphism} is an orientation-preserving diffeomorphism ${\phi\colon X^+ \to X^-}$ together with a choice of lift of the induced map on $\GL^+(\R)$-frame bundles to the topological spin bundles.

If $E^\pm\ra X$ are $\SU(m)$-bundles and $\Phi:E^+\ra E^-$ a $\SU(m)$-isomorphism, we may take $X^+=X^-=U^+=U^-=X,$ and $\phi=\id_X,$ and $\rho^+=\es=\rho^-$ to be defined over the empty set. Then we use the shorthand
\begin{equation*}
\Or(\Phi)=\Or(\id_X,\Phi,\es,\es):\Or_{E^+}\longra\Or_{E^-}.	
\end{equation*}

\subsection{Global automorphisms}

\subsubsection{Mapping torus}

Theorem~\ref{thm.2.15} includes as the special case $U^\pm = X^\pm$ the more obvious functoriality for globally defined diffeomorphisms $\phi\colon X^+ \to X^-$ and $\Phi.$ The theorem can be regarded as extending this functoriality to open manifolds and compactly supported data. The effect of a globally defined diffeomorphism can be studied using the following construction.

\begin{dfn} The \emph{mapping torus} of a diffeomorphism $\psi\colon X \to X$ is the quotient $X_\psi$ of $X\times [0,1]$ by the equivalence relation $(x,1)\sim (\psi(x),0).$
\end{dfn}

\begin{prop} The mapping torus has the following properties:
\begin{itemize}
\setlength{\itemsep}{0pt}
\setlength{\parsep}{0pt}
\item[{\bf(i)}] If\/ $X$ is compact, then\/ $X_\psi$ is compact.
\item[{\bf(ii)}] $X_\psi$ is a fibre bundle over\/ $\cS^1$ with typical fibre $X.$
\item[{\bf(iii)}] If\/ $X$ is oriented and\/ $\psi$ is orientation preserving, then\/ $X_\psi$ is oriented.
\item[{\bf(iv)}] When\/ $X$ has a spin structure and\/ $\psi$ is a spin structure preserving diffeomorphism we get a topological spin structure on\/ $X_\psi.$
\item[{\bf(v)}] Let\/ $E\to X$ be a vector bundle and\/ $\Psi\colon E \to E$ an automorphism covering\/ $\psi.$ Then the mapping torus\/ $E_\Psi$ is a vector bundle over\/ $X_\psi.$
\end{itemize}
\end{prop}

\subsubsection{Calculating the effect on orientations using the mapping torus}

\begin{prop}\label{prop.2.18}
Let\/ $X$ be an odd-dimensional compact spin manifold and\/
$\Psi\colon E \to\psi^*(E)$ a\/ $\SU(m)$-isomorphism of a\/ $\SU(m)$-bundle\/ $E\to X$ covering a spin diffeomorphism\/ ${\psi\colon X \to X}.$ Then\/ $\Or(\psi,\Psi,\es,\es) = (-1)^{\delta(\psi,\Psi)}\cdot \id_{\Or_E}$ for
\begin{equation*}
\delta(\psi,\Psi)
\coloneqq 
\int_{X_\psi} \hat{A}(TX_\psi)\bigl( \mathrm{ch}(E_\Psi^* \otimes E_\Psi)-\rk(E_\Psi)^2 \bigr).
\end{equation*}
\end{prop}

\begin{proof} By choosing a connection $A_0 \in \A_E$ and any smooth path $A_t,$ $t\in[0,1]$ from $A_0$ to $A_1=\Psi^*A_0$ we may regard $E_\Psi \to X_\psi$ as a $\cS^1$-family of $\SU(m)$-bundles with connection. Pick a metric on $X_\psi.$ Using the induced metrics $g_t$ on $X_t$ we can form the $\cS^1$-family of Diracians $\D^{g_t}_{\Ad A_t}$ twisted by $\Ad E$ and the $\cS^1$-family $\D^{g_t}_{\Ad \underline\C^m, 0}$ twisted by the flat connection. Then $\Or\bigl(\D^{g_t}_{\Ad A_t} \bigr)\otimes_{\Z_2} \Or\bigl(\D^{g_t}_{\Ad \underline\C^m, 0}\bigr)^*$ is a double cover of $\cS^1$ with monodromy $\Or(\psi,\Psi,\es,\es).$ On the other hand, since $\dim X$ is odd, the Diracians are self-adjoint and the monodromy coincides with the spectral flow around the loop~\cite[Th.~7.4]{APS2}. 

As explained by Atiyah--Patodi--Singer in \cite[p.~95]{APS2}, the spectral flow around a loop agrees with the index of a single operator on the mapping torus, using \cite[Th.~3.10]{APS1}. For the family $\D^{g_t}_{\Ad A_t}$ we get the positive Diracian $\D^+$ on $X_\psi$ twisted by $\Ad E_\Psi.$ To compute the index of a single operator we may complexify and can then apply the cohomological index formula of Atiyah--Singer \cite{AtSi3} to get
\begin{equation*}
\ind (\D^+_{\Ad E_\Psi})\! =\! \int_{X_\psi} \!\!\!\hat{A}(TX_\psi) \mathrm{ch}( \Ad E_\Psi\otimes_\R \C)
\!=\! \int_{X_\psi} \!\!\!\hat{A}(TX_\psi) \bigl(\mathrm{ch}(E_\Psi^* \otimes E_\Psi) - 1\bigr),
\end{equation*}
using $(\Ad E_\Psi\otimes_\R \C) \oplus \C \cong E_\Psi^* \otimes E_\Psi.$ Applying the
same argument to the family $\D^{g_t}_{\Ad \underline\C^m, 0}$ and subtracting yields the desired result.
\end{proof}

\begin{prop}
\label{prop.2.19}
For\/ $\psi,\Psi$ as in Proposition \textup{\ref{prop.2.18}}
and\/ $\dim X = 7$ we have
\begin{equation}
\label{equation.2.6}
\delta(\psi,\Psi) \equiv \int_{X_\psi} c_2(E_\Psi)^2 \equiv \frac12 \int_{X_\psi} p_1(TX_\psi) c_2(E_\Psi)
\mod 2.
\end{equation}
Hence\/ $\Or(\Psi)\coloneqq\Or(\id,\Psi,\es,\es)=\id$ for every\/
$\SU(m)$-automorphism\/ $\Psi:E\ra E$ \textup(this was 
obtained by Walpuski in \textup{\cite[Prop.~6.3]{Walp2}}\textup). Therefore\/ $\Or(\psi,\Psi_1,\es,\es)=\Or(\psi,\Psi_2,\es,\es)$ whenever\/ $\Psi_1$ and\/ $\Psi_2$ cover the same spin diffeomorphism.
\end{prop}

\begin{proof} The proof is similar to that of Walpuski \cite[Prop.~6.3]{Walp2}. We have
\begin{equation*}
\delta(\psi,\Psi) = \int_{X_\psi} \frac{m+6}{6}c_2(E_\Psi)^2 - \frac{m}{3}c_4(E_\Psi) + \frac{m}{12} p_1(TX_\psi)c_2(E_\Psi).
\end{equation*}
By the Atiyah--Singer index theorem \cite{AtSi3}
\begin{equation*}
A\coloneqq \int_{X_\psi} \hat{A}(TX_\psi)\bigr( \mathrm{ch}(E_\Psi) - m \bigr)
=
\int_{X_\psi} \frac{c_2(E_\Psi)^2 - 2c_4(E_\Psi)}{12} + \frac{p_1(TX_\psi)c_2(E_\Psi)}{24}
\end{equation*}
is an index and hence an integer. Then \eqref{equation.2.6} follows from
\begin{align*}
\delta(\psi,\Psi) - 2m\cdot A &= \int_{X_\psi} c_2(E_\Psi)^2,\\
\delta(\psi,\Psi) - (12+2m)A &= 2\int_{X_\psi} c_4(E_\Psi) + \frac12 p_1(TX_\psi)c_2(E_\Psi).
\end{align*}
Finally, when $\psi=\id_X$ we have $X_\psi = X\times \cS^1$ and $p_1(TX_\psi)=\pr_X^*p_1(TX).$
On the spin $7$-manifold $X$ the cohomology class $p_1(TX)$ is divisible by four.
\end{proof}

\begin{ex}
\label{ex.2.20}
Let $E \to X$ be a $\SU(m)$-bundle over a compact spin $7$-manifold
with second Chern class Poincar\'e dual to a $3$-submanifold $Y \subset X.$

Let $\Psi\colon E \to E$ be a $\SU(m)$-isomorphism covering a spin
diffeomorphism $\psi\colon X \to X$ satisfying $\psi|_Y = \id_Y.$ Then
we may regard $Y\times \cS^1 \subset X_\psi.$

Formula \eqref{equation.2.6} is the self-intersection (mod $2$) of the
class $\al$ in $H_4(X_\psi)$ Poincar\'e dual to $c_2(E_\Psi).$ We
have $\al = [Y\times \cS^1] + \be$ for some $\be \in H_4(X),$ where $X$ is included into $X_\psi$
at some fixed point of $[0,1].$ As the cross term appears twice and $\be \bullet \be = 0$ in $X_\psi,$ we get
\begin{equation*}
\delta(\psi,\Psi) \equiv \int_{X_\psi} c_2(E_\Psi)^2 \equiv \al\bullet \al
\equiv [Y\times \cS^1]\bullet [Y\times \cS^1] \mod 2.
\end{equation*}
This again shows that $\delta(\id,\Psi)\equiv 0$ when $\psi = \id.$
\end{ex}

\subsection{\texorpdfstring{An example of a non-orientable moduli space $\cB_Q$}{An example of a non-orientable moduli space ℬᵨ}}
\label{subsection.2.4}

Suppose $(X,g)$ is a compact, spin Riemannian 7-manifold, and $G$ is any Lie group, and $Q\ra X$ is a principal $G$-bundle. Then generalizing \S\ref{section.1} we may define moduli spaces $\A_Q$ of connections on $Q$ and $\cB_Q=\A_Q/\G$ of connections on $Q$ modulo gauge, and a principal $\Z_2$-bundle $\bar O^{\D^{\smash{g}}}_Q\ra \A_Q$ parametrizing orientations on $\det\D^g_{\Ad A}$ for~$A\in\A_Q.$

In \S\ref{section.1} we took $G=\SU(m)$ or $\U(m),$ and then Walpuski \cite[Prop.~6.3]{Walp2} and \cite[Ex.~2.13]{JTU} show that $\G$ acts trivially on the set of global sections of $\bar O^{\D^{\smash{g}}}_Q,$ so that $\bar O^{\D^{\smash{g}}}_Q$ descends to a principal $\Z_2$-bundle $O^{\D^{\smash{g}}}_Q\ra\cB_Q,$ which is orientable. But what about other Lie groups~$G$?

This section will give an example of $(X,g)$ for which when $G=\Sp(m)$ for $m\ge 2$ and $Q=X\t\Sp(m)\ra X$ is the trivial $\Sp(m)$-bundle, $\G$ acts non-trivially on the set of global sections of $\bar O^{\D^{\smash{g}}}_Q,$ so that although $\bar O^{\D^{\smash{g}}}_Q$ does in fact descend to a principal $\Z_2$-bundle $O^{\D^{\smash{g}}}_Q\ra\cB_Q,$ this is non-orientable (i.e.\ it has no global sections). Hence the analogue of \cite[Prop.~6.3]{Walp2} is {\it false\/} for $\Sp(m)$-bundles.

A result on stabilizing $\H^m$-bundles \cite[Ex.~2.16]{JTU} implies that if $O^{\D^{\smash{g}}}_Q\ra\cB_Q$ is non-orientable for $Q=X\t\Sp(2)$ the trivial $\Sp(2)$-bundle, then the same holds for $Q=X\t\Sp(m)$ for $m\ge 2.$ So we consider only~$G=\Sp(2).$

To show that $O^{\D^{\smash{g}}}_Q\ra\cB_Q$ is non-orientable, it is enough to find a smooth loop $\ga:\cS^1\ra\cB_Q$ such that the monodromy of $O^{\D^{\smash{g}}}_Q$ around $\ga$ is $-1.$ As in \cite[\S 2.2]{CGJ}, a smooth loop $\ga:\cS^1\ra\cB_Q$ is equivalent to a principal $\Sp(2)$-bundle $R\ra X\t\cS^1$ which is trivial on $X\t\{1\},$ together with a partial connection on $R$ in the $X$ directions, and any such $R$ may be written as the mapping torus $R_f$ of a smooth map $f:X\ra\Sp(2),$ obtained by taking the trivial bundle on $X\times [0,1]$ and identifying endpoints using the gauge transformation~$f\colon X \to \Sp(2).$

As in Walpuski \cite[\S 6.1]{Walp2} or in Proposition~\ref{prop.2.18}, the monodromy of $O^{\D^{\smash{g}}}_Q$ around $\ga$ is $(-1)^{\mathop{\rm SF}(\ga)},$ where $\mathop{\rm SF}(\ga)$ is the spectral flow of the family of elliptic operators $\bigl(\D^g_{\Ad\ga(t)}\bigr)_{t\in\cS^1},$ which may be computed as an index $\mathop{\rm SF}(\ga)=\ind(\D^+_{\Ad(R_f)})$ of the positive Dirac operator $\D^+$ on $X\t\cS^1$ twisted by any connection on~$\Ad(R_f).$ 

Thus, to show that $O^{\D^{\smash{g}}}_Q\ra\cB_Q$ is non-orientable on $X,$ we should find a compact spin 7-manifold $X$ and a smooth $f:X\ra\Sp(2)$ such that $\ind(\D^+_{\Ad(R_f)})$ is odd. We will do this in Example \ref{ex.2.24}, after some initial computations.

\begin{lem} For a\/ $\Sp(2)$-bundle\/ $R$ over an\/ $8$-dimensional base we have
\begin{equation}
\label{equation.2.7}
\smash{\operatorname{ch}(\Ad(R)\otimes_\R \C) = 10 - 6c_2(R) + \frac{3}{2} c_2(R)^2 - 2c_4(R).}
\end{equation}
\end{lem}

\begin{proof} This can be computed using Chern roots, meaning it
suffices to establish \eqref{equation.2.7} in the case that the $\H^2$-bundle $(R\t\H^2)/\Sp(2)\ra X$ associated to $R$ is the direct sum of quaternionic line bundles.
\end{proof}

\begin{prop} Let\/ $X$ be a compact spin\/ $7$-manifold and\/ $R\to X\times \cS^1$ a\/ $\Sp(2)$-bundle. Then the index has the parity of the Euler number of\/ $R$:
\begin{equation*}
\ind\, \D^+_{\Ad(R)} \equiv \int_{X\times \cS^1} c_4(R) \mod 2.
\end{equation*}
\end{prop}

\begin{proof} Using \eqref{equation.2.7} we find that
\begin{equation*}
\ind\, \D^+_{\Ad(R)} \,+\, 6\cdot\ind\, \D^+_R
= \int_{X\times \cS^1} \frac{p_1(TX)c_2(R)}{2} + 2c_2(R)^2 - 3c_4(R).
\end{equation*}
As $p_1(TX)$ is divisible by four, the first summand on the right is even.
\end{proof}

Let $R_f$ be the mapping torus bundle over $X\times \cS^1$ of a smooth $f\colon X \to \Sp(2).$ Then the Euler number of $R_f$ is the degree of
\begin{equation*}
X\,{\buildrel f\over\longra}\, \Sp(2) \,{\buildrel \pi\over\longra}\, \Sp(2)/\Sp(1) = \cS^7.
\end{equation*}
For non-orientability, we seek $X$ and $f$ such that this degree is odd.

\begin{ex} Let $X=\cS^7$ and $f\colon \cS^7 \to \Sp(2)$ be smooth. Then the degree of $\pi\circ f$ is always divisible by $12$ and therefore $O^{\D^{\smash{g}}}_Q\ra\cB_Q$ is orientable, for $Q$ the trivial $\Sp(2)$-bundle over $\cS^7.$ To see this, consider the long exact sequence of homotopy groups of the fibration~$\pi$:
\begin{equation*}
\xymatrix@C=23pt{
\cdots \ar[r]&\pi_7 \Sp(2)\ar[r]^{\pi_*} & \pi_7 \cS^7 = \Z\ar[r]^\partial & \pi_6 \Sp(1)\ar[r] & \pi_6 \Sp(2)\ar[r]
& \cdots }
\end{equation*}
As $\pi_6 \Sp(1) = \Z_{12}$ and $\pi_6 \Sp(2) = \{0\},$ the cokernel of $\pi_*$ is $\Z_{12}.$
Hence orientability holds for the moduli space of $\Sp(2)$-connections on $\cS^7.$
\end{ex}

\begin{ex}
\label{ex.2.24}
For $\bigl(\begin{smallmatrix} a & b\\ c & d \end{smallmatrix}\bigr)\in \Sp(2)$ and $q\in \Sp(1),$ define
\begin{equation*}
M \coloneqq \begin{pmatrix}
|a|^2 + bq\bar{b} & a\bar{c} + bq\bar{d}\\
c\bar{a}+dq\bar{b} & |c|^2 + dq\bar{d}
\end{pmatrix}.
\end{equation*}
Then $M\in \Sp(2).$ Replacing $(a,b,c,d,q)$ by $(ar,bs,cr,ds,\bar{s}qs)$ for $r,s\in \Sp(1)$ does not change the matrix $M.$
Hence the formula defines a map
\begin{equation*}
f\colon X= \Sp(2)\times_{\Sp(1)\times \Sp(1)} \Sp(1) \longra \Sp(2),
\end{equation*}
where the diagonal subgroup $(r,s)\in \Sp(1)\times \Sp(1)\subset\Sp(2)$ acts on $q\in \Sp(1)$ by conjugating with the second factor $q\mapsto sq\bar{s}.$ It is easy to see that ${\operatorname{deg}(\pi\circ f) = 1}.$ It follows that $O^{\D^{\smash{g}}}_Q\ra\cB_Q$ is non-orientable, where $Q=X\t\Sp(2)\ra X$ is the trivial $\Sp(2)$-bundle over $X.$
\end{ex}

In the sequel \cite[Ex.~1.14]{CGJ} we will use Example \ref{ex.2.24} to find non-orientable moduli spaces of $\Sp(m)$-connections $\cB_Q$ for $m\ge 2$ on the 8-manifold~$X\t\cS^1.$

\section{Flag structures}
\label{section.3}

We recall the following from~\cite[\S 3.1]{Joyc3}. Here $X$ is not assumed to be compact.

\begin{dfn}
Let $X$ be an oriented 7-manifold, and consider pairs $(Y,s)$ of a compact, oriented $3$-submanifold $Y\subset X,$ and a non-vanishing section $s$ of the normal bundle $N_Y$ of $Y$ in $X.$
We call $(Y,s)$ a {\it flagged submanifold\/} in~$X.$

For non-vanishing sections $s, s'$ of $N_Y$ define
\begin{equation}
d(s,s') \coloneqq Y \bullet \bigl\{ t\cdot s(y) + (1-t)\cdot s'(y) \enskip\big|\enskip t\in [0,1],\enskip y\in Y\bigr\} \in \Z,
\label{equation.3.1}
\end{equation}
using the intersection product `$\bu$' between a $3$-cycle and a $4$-chain
whose boundary does not meet the cycle, see Dold \cite[(13.20)]{Dold}, where we identify $Y$ with the zero section in~$N_Y.$

Let $(Y_1,s_1),(Y_2,s_2)$ be disjoint flagged submanifolds with $[Y_1]=[Y_2]$ in $H_3(X;\Z).$ Choose an integral 4-chain $C$ with $\pd C=Y_2-Y_1.$ Let $Y_1',Y_2'$ be small perturbations of $Y_1,Y_2$ in the normal directions $s_1,s_2.$ Then $Y_1'\cap Y_1=Y_2'\cap Y_2=\es$ as $s_1,s_2$ are non-vanishing, and $Y_1'\cap Y_2=Y_2'\cap Y_1=\es$ as $Y_1,Y_2$ are disjoint and $Y_1',Y_2'$ are close to $Y_1,Y_2.$ Define $D((Y_1,s_1),(Y_2,s_2))$ to be the intersection number $(Y_2'-Y_1')\bu C$ in homology over $\Z.$ Here we regard
\begin{equation*}
[C] \in H_4(X,Y_1\cup Y_2;\Z),\qquad
[Y'_1], [Y'_2] \in H_3(Y'_1\cup Y'_2,\emptyset;\Z).
\end{equation*}
Note that since $Y'_1, Y'_2$ are small perturbations and $Y_1, Y_2$ are disjoint we have $(Y_1 \cup Y_2) \cap (Y'_1 \cup Y'_2) = \emptyset.$
This is independent of the choices of $C$ and~$Y_1',Y_2'.$
\end{dfn}

In \cite[Prop.s~3.3 \& 3.4]{Joyc3} the first author shows that if $(Y_1,s_1),(Y_2,s_2),\ab(Y_3,s_3)$ are disjoint flagged submanifolds with $[Y_1]=[Y_2]=[Y_3]$ in $H_3(X;\Z)$ then
\e
\label{equation.3.2}
\begin{split}
D((Y_1,s_1),(Y_3,s_3))&\equiv D((Y_1,s_1),(Y_2,s_2))\\
&\qquad +D((Y_2,s_2),(Y_3,s_3))\mod 2,
\end{split}
\e
and if $(Y',s')$ is any small deformation of $(Y,s)$ with $Y,Y'$ disjoint then
\e
D((Y,s),(Y',s'))\equiv 0\mod 2.
\label{equation.3.3}
\e

\begin{dfn}
\label{dfn.3.2}
A {\it flag structure\/} on $X$ is a map
\begin{equation*}
F:\bigl\{\text{flagged submanifolds $(Y,s)$ in $X$}\bigr\}\longra\{\pm 1\},
\end{equation*}
satisfying:
\begin{itemize}
\setlength{\itemsep}{0pt}
\setlength{\parsep}{0pt}
\item[(i)] $F(Y,s) = F(Y,s')\cdot (-1)^{d(s,s')}.$
\item[(ii)] If $(Y_1,s_1),(Y_2,s_2)$ are disjoint flagged submanifolds in $X$ with $[Y_1]=[Y_2]$ in $H_3(X;\Z)$ then
\begin{equation*}
F(Y_2,s_2)=F(Y_1,s_1)\cdot (-1)^{D((Y_1,s_1),(Y_2,s_2))}.
\end{equation*}
This is a well behaved condition by \eq{equation.3.2}--\eq{equation.3.3}.
\item[(iii)] If $(Y_1,s_1),(Y_2,s_2)$ are disjoint flagged submanifolds then
\end{itemize}
\begin{equation*}
F(Y_1\amalg Y_2,s_1\amalg s_2)=F(Y_1,s_1)\cdot F(Y_2,s_2).
\end{equation*}
Flag structures restrict to open subsets in the obvious way.
\end{dfn}

Here is \cite[Prop.~3.6]{Joyc3}:

\begin{prop}
\label{prop.3.3}
Let\/ $X$ be an oriented\/ $7$-manifold. Then:
\begin{itemize}
\setlength{\itemsep}{0pt}
\setlength{\parsep}{0pt}
\item[{\bf(a)}] There exists a flag structure\/ $F$ on\/ $X.$
\item[{\bf(b)}] If\/ $F,F'$ are flag structures on\/ $X$ then there exists a unique group morphism\/ $H_3(X;\Z)\ra\{\pm 1\},$ denoted\/ $F'/F,$ such that
\e
F'(Y,s)=F(Y,s)\cdot (F'/F)[Y] \qquad\text{for all\/ $(Y,s).$}
\label{equation.3.4}
\e
\item[{\bf(c)}] Let\/ $F$ be a flag structure on\/ $X$ and\/ $\ep:H_3(X;\Z)\ra\{\pm 1\}$ a morphism, and define\/ $F'$ by \eq{equation.3.4} with\/ $F'/F=\varepsilon.$ Then\/ $F'$ is a flag structure on\/~$X.$
\end{itemize}
Hence the set of flag structures on\/ $X$ is a torsor over\/ $\Hom\bigl(H_3(X;\Z),\Z_2\bigr).$
\end{prop}

\begin{ex} Every oriented 7-manifold $X$ with $H_3(X;\Z)=0$ has a unique flag structure. More generally, a basis $[Y_i]$ of the image of $H_3(X;\Z)\!\to\! H_3(X;\Z_2)$ for submanifolds $Y_i\subset X$ with chosen normal sections $s_i$ induce a unique flag structure $F$ with $F(Y_i,s_i) = 1.$ For example, $\cS^7$ has a unique flag structure $F_{\cS^7},$ and $Y^3\times \cS^4$ has a preferred flag structure $F$ with $F(Y\t\{x\},Y\t v)=1$ for any $x\in\cS^4$ and~$0\ne v\in T_x\cS^4.$
\end{ex}

\begin{dfn} Let $\psi \colon X' \to X$ be an orientation-preserving diffeomorphism. The \emph{pullback} of the flag structure $F$ on $X$
is $(\psi^*F)(Y',s') \coloneqq F(\psi(Y'), \d\psi\circ s').$ The pushforward is defined to be the
pullback along $\psi^{-1}.$
\end{dfn}

When $X'=X$ we can compare a flag structure to its pullback along $\psi$:

\begin{prop}
\label{prop.3.6}
Let\/ $X$ be an oriented\/ $7$-manifold and\/ $Y\subset X$ a compact oriented\/ $3$-submanifold. Suppose\/ $\psi\colon X\to X$ is an orientation-preserving diffeomorphism with\/ $\psi|_Y = \id_Y.$ Then\/ $(F/\psi^*F)[Y]=(-1)^{[Y\t\cS^1]\bu[Y\t\cS^1]}$ for any flag structure\/ $F$ on\/ $X,$ where\/ $[Y\t\cS^1]\bu[Y\t\cS^1]$ is the self-intersection of\/ $Y\times \cS^1$ in the mapping torus\/ $X_\psi.$
\end{prop}

\begin{proof} Pick $s\colon Y \to N_Y$ non-vanishing.
Let ${\Gamma(y,t) \coloneqq (1-t)s(y) + t\cdot \d\psi\circ s(y)}$
for $y\in Y$ and $t\in [0,1].$ By Proposition \ref{prop.3.3}(b), Definition \ref{dfn.3.2}(i) and \eqref{equation.3.1}
\begin{equation*}
(F/\psi^*F)[Y] = F(Y,s) \cdot F(Y,\d\psi\circ s)^{-1} = (-1)^{d(s,\d\psi\circ s)} = (-1)^{Y\bullet \Im(\Gamma)}.
\end{equation*}
The normal bundle of $Y\times \cS^1$ in $X_\psi$
is the mapping torus of $\d\psi\colon N_Y \to N_Y,$ so
we can regard $\Gamma$ as a normal section of $Y\times \cS^1$ in $X_\psi$ by
$\hat\Gamma(y,t) \coloneqq [\Gamma(y,t),t].$ There being no intersection points at $t=0,1$ we have
${Y\bullet \Im(\Gamma)} = {(Y\times \cS^1)\bullet \Im(\hat\Gamma)}.$
Finally, by taking $s$ to zero $\hat\Gamma$ is homologous to $Y\times \cS^1$ in $X_\psi.$
\end{proof}

\section{Canonical orientations}
\label{section.4}

This section proves our main result Theorem \ref{thm.1.2}, following the outline in~\S\ref{section.1}.

\renewcommand\thesubsection{\thesection(\Alph{subsection})}
\renewcommand\thesubsubsection{\thesection(\Alph{subsection})(\roman{subsubsection})}

\subsection{The construction of orientations $o^F_{Y,\rho,\io,\Psi}(E)$}

\subsubsection{Stably trivial bundles}

We first extend $o^\mathrm{flat}(\underline{\C}^m) \in \Or_{\underline{\C}^m}$ from Definition~\ref{dfn.2.13} to stably trivial bundles.

\begin{dfn}
\label{dfn.4.1}
Let $E \to X$ be a stably trivial $\SU(m)$-bundle over a compact spin $7$-manifold. Then we find a $\SU(m+k)$-isomorphism $\Phi\colon E\oplus \underline{\C}^k \to \underline{\C}^{m+k}$
over $\id_X.$ Using \eqref{equation.2.3} and $\Or(\Phi)$
we can identify $\Or_E$ with $\Z_2.$ That is,
there exists a unique $o^\mathrm{flat}(E) \in \Or_E$ satisfying
\e
\label{equation.4.1}
\Or(\Phi)\bigl(
o^\mathrm{flat}(E) \cdot o^\mathrm{flat}(\underline\C^k)
\bigr)
=  o^\mathrm{flat}(\underline\C^{m+k}),
\e
using the product notation defined after Proposition~\ref{prop.2.14}.
\end{dfn}

\begin{prop}
\label{prop.4.2}
These orientations have the following properties:
\begin{itemize}
\setlength{\itemsep}{0pt}
\setlength{\parsep}{0pt}
\item[{\bf(i)}]\textup{(Well-defined.)}~The definition of\/ $o^\mathrm{flat}(E)$ is independent of\/ $k$ and\/ $\Phi.$
\item[{\bf(ii)}]\textup{(Families.)}~Let\/ $P$ be compact Hausdorff and\/ $E\to X\times P$ stably trivial. Then\/ $p\mapsto o^\mathrm{flat}(E|_{X\times \{p\}})$ is a continuous section of the double cover\/~$\Or_E.$
\item[{\bf(iii)}]\textup{(Functoriality.)}~Let\/ $E_1$ be a\/ $\SU(m_1)$-bundle and\/ $E_2$ a $\SU(m_2)$-bundle over\/ $X,$ both stably trivial. Let\/ $\ell_1, \ell_2 \in \N$ and let\/ ${\Psi\colon E_1 \oplus \underline\C^{\ell_1} \to E_2 \oplus \underline\C^{\ell_2}}$ be a\/ $\SU$-isomorphism. Then\/ $\Or(\Psi)(o^\mathrm{flat}(E_1\oplus \underline\C^{\ell_1}))=o^\mathrm{flat}(E_2\oplus \underline\C^{\ell_2}).$
\item[{\bf(iv)}]\textup{(Stability.)}~$o^\mathrm{flat}(E\oplus \underline\C^\ell)=o^\mathrm{flat}(E) \cdot o^\mathrm{flat}(\underline\C^\ell).$
\item[{\bf(v)}]\textup{(Direct sums.)}~Let\/ $E_1,E_2\ra X$ be\/ $\SU(m_1)$- and\/ $\SU(m_2)$-bundles, both stably trivial. Then $o^\mathrm{flat}(E_1 \oplus E_2) = o^\mathrm{flat}(E_1) \cdot o^\mathrm{flat}(E_2).$
\end{itemize}
\end{prop}

\begin{proof}(i) Given $\ell \in \N$ and $\Phi\colon E\oplus \underline\C^k \to \underline\C^{m+k}$ we use the properties of Proposition~\ref{prop.2.14}
and find
\begin{align*}
o^\mathrm{flat}(\underline\C^{m+k+\ell})
&=
o^\mathrm{flat}(\underline\C^{m+k}) \cdot o^\mathrm{flat}(\underline\C^{\ell}) &&\text{(unital)}\\
&=
\Or(\Phi)\bigl( o^\mathrm{flat}(E)\cdot o^\mathrm{flat}(\underline\C^k) \bigr)\cdot o^\mathrm{flat}(\underline\C^{\ell}) &&\text{(by \eqref{equation.4.1})}\\
&=
\Or(\Phi\oplus\id_{\underline\C^\ell})\bigl[\bigl( o^\mathrm{flat}(E)\cdot o^\mathrm{flat}(\underline\C^k) \bigr)\cdot o^\mathrm{flat}(\underline\C^{\ell})\bigr] &&\text{(by \eqref{equation.2.5})}\\
&=
\Or(\Phi\oplus\id_{\underline\C^\ell})\bigl[ o^\mathrm{flat}(E)\cdot \bigl(o^\mathrm{flat}(\underline\C^{k})\cdot o^\mathrm{flat}(\underline\C^{\ell})\bigr) \bigr]&&\text{(assoc.)}\\
&=
\Or(\Phi\oplus\id_{\underline\C^\ell})\bigl( o^\mathrm{flat}(E)\cdot o^\mathrm{flat}(\underline\C^{k+\ell}) \bigr)&&\text{(unital)}.
\end{align*}
This proves independence of $k.$ The independence of $\Phi$
follows from Proposition~\ref{prop.2.19}. Once $o^\mathrm{flat}$ is
well-defined, (iii)--(iv) are clear.\medskip

\noindent (ii) We know already that
$\Or(\Phi)$ and $\lambda$ are continuous maps of double covers over $P$
and that $o^\mathrm{flat}(\underline\C^k)$ and $o^\mathrm{flat}(\underline\C^{m+k})$ are continuous sections.
\medskip

\noindent (v) Let $\Phi_i\colon E_1 \oplus \underline\C^{k_i} \to \underline\C^{m_i + k_i}$ be $\SU(m_i + k_i)$-isomorphisms. Then
\begin{align*}
&o^\mathrm{flat}(\underline\C^{m_1 + m_2 + k_1 + k_2})\\
&= o^\mathrm{flat}(\underline\C^{m_1+ k_1})\cdot o^\mathrm{flat}(\underline\C^{m_2+ k_2})  && (\text{unital})\\
&=\Or(\Psi_1)\bigl(o^\mathrm{flat}(E_1)\cdot o^\mathrm{flat}(\underline\C^{k_1}) \bigr) \cdot
\Or(\Psi_2)\bigl(o^\mathrm{flat}(E_2)\cdot o^\mathrm{flat}(\underline\C^{k_2}) \bigr) &&\text{(by \eqref{equation.4.1})}\\
&= \Or(\Psi_1 \oplus \Psi_2)\bigl( o^\mathrm{flat}(E_1)\cdot o^\mathrm{flat}(\underline\C^{k_1}) \cdot
o^\mathrm{flat}(E_2)\cdot o^\mathrm{flat}(\underline\C^{k_2}) \bigr)&&\text{(by \eqref{equation.2.5})}\\
&= \Or(\Psi_1 \oplus \Psi_2)\Or(\id_{E_1}\oplus \operatorname{flip} \oplus \id_{\underline\C^{k_2}})\\ &\qquad\qquad\qquad\qquad\bigl( o^\mathrm{flat}(E_1)\cdot o^\mathrm{flat}(E_2) \cdot o^\mathrm{flat}(\underline\C^{k_1}) \cdot
o^\mathrm{flat}(\underline\C^{k_2}) \bigr)&&\text{(comm.).}
\end{align*}
On the other hand, by using the trivialization $(\Phi_1\oplus \Phi_2)\circ(\id_{E_1}\oplus \operatorname{flip} \oplus \id_{\underline\C^{k_2}})$ of $E_1\oplus E_2 \oplus \underline\C^{k_1 + k_2}$ in the definition \eqref{equation.4.1} we have
\begin{align*}
\Or(\Phi_1\oplus \Phi_2)&\circ \Or(\id_{E_1}\oplus \operatorname{flip} \oplus \id_{\underline\C^{k_2}})(o^\mathrm{flat}(E_1\oplus E_2) \cdot o^\mathrm{flat}(\underline\C^{k_1 + k_2}))\\
&= o^\mathrm{flat}(\underline\C^{m_1 + m_2 + k_1 + k_2}).
\end{align*}
Combining the last two equations and using unitality implies the result.
\end{proof}

\begin{ex}
\label{ex.4.3}
Since $\pi_6(\SU(m)) = 0$ for $m\geq 4,$ every $\SU(m)$-bundle $E$ over $\cS^7$ is stably trivial. 
\end{ex}

\subsubsection{Trivializing $\SU(m)$-bundles outside codimension 4}

We now explain how to trivialize a $\SU(m)$-bundle $E\ra X$ outside a submanifold $Y\subset X$ of codimension 4. One might expect this on general grounds, as $\pi_i(B\SU(m))=0$ for $i<4,$ but the construction may be of independent interest.

\begin{con}
\label{con.4.4}
Let $X$ be a compact, oriented manifold of dimension $n,$ with $n\le 11,$ and $E\ra X$ be a rank $m$ complex vector bundle with $\SU(m)$-structure, for $m\ge 1.$ Write $\underline\C^{m-1}=X\t\C^{m-1}$ for the trivial vector bundle over $X$ with fibre $\C^{m-1},$ and $\Hom(\underline\C^{m-1},E)\ra X$ for the bundle of complex vector bundle morphisms over $X,$ and $\Hom^{(k)}(\underline\C^{m-1},E)$ for the determinantal variety of 
homomorphisms ${s\colon \C^{m-1} \to E_x}$ of rank $m-1-k$ for $k=0,\ldots,m-1,$ which is a submanifold of $\Hom(\underline\C^{m-1},E)$ of real codimension $2k(k+1).$

A morphism $s:\ul\C^{m-1}\ra E$ is a section $s\colon X \to \Hom(\underline\C^{m-1},E).$ We call $s$ \emph{generic} if it is transverse to each $\Hom^{(k)}(\underline\C^{m-1},E).$ This is an open dense condition on such $s.$ If $s$ is generic then $s^{-1}\bigl(\Hom^{(k)}(\underline\C^{m-1},E)\bigr)$ is a submanifold of $X$ of dimension $n-2k(k+1),$ and so is empty if $k\ge 2$ as $n\le 11.$ 

Write $Y=s^{-1}\bigl(\Hom^{(1)}(\underline\C^{m-1},E)\bigr),$ the {\it degeneracy locus\/} of $s.$ Then $Y$ is an embedded submanifold of $X$ of dimension $n-4.$ It is closed in $X,$ as the closure lies in the union of $s^{-1}\bigl(\Hom^{(k)}(\underline\C^{m-1},E)\bigr)$ for $k\ge 1,$ but these are empty for $k>1.$ It is oriented as $X$ is, and the fibres of $\Hom(\underline\C^{m-1},E)\ra X$ and $\Hom^{(1)}(\underline\C^{m-1},E)\ra X$ are complex manifolds and so oriented.

As $X\sm Y=s^{-1}\bigl(\Hom^{(0)}(\underline\C^{m-1},E)\bigr),$ we see that $s\vert_{X\sm Y}:\underline\C^{m-1}\vert_{X\sm Y}\ra E\vert_{X\sm Y}$ is injective. Let $s^*\vert_{X\sm Y}:E\vert_{X\sm Y}\ra\underline \C^{m-1}\vert_{X\sm Y}$ be its Hermitian adjoint, with respect to the Hermitian metrics on $E$ and $\ul\C^{m-1}.$ Then $s^*\ci s\vert_{X\sm Y}:\underline \C^{m-1}\vert_{X\sm Y}\ra \underline \C^{m-1}\vert_{X\sm Y}$ is invertible with positive eigenvalues, and thus has an inverse square root $(s^*\ci s)\vert_{X\sm Y}^{-1/2}:\underline \C^{m-1}\vert_{X\sm Y}\ra \underline \C^{m-1}\vert_{X\sm Y}.$ Consider
\begin{equation*}
s\vert_{X\sm Y}\ci (s^*\ci s)\vert_{X\sm Y}^{-1/2}:\underline \C^{m-1}\vert_{X\sm Y}\longra E\vert_{X\sm Y}.
\end{equation*}
This is an injective linear map of complex vector bundles which is isometric for the Hermitian metrics on $\underline \C^{m-1}\vert_{X\sm Y}$ and $E\vert_{X\sm Y}.$ 

As $E$ has a $\SU(m)$-structure, there is a unique isomorphism of $\SU(m)$-bundles $\rho:\underline \C^m\vert_{X\sm Y}\ra E\vert_{X\sm Y},$ such that $\rho\vert_{\underline \C^{m-1}\vert_{X\sm Y}}=s\vert_{X\sm Y}\ci (s^*\ci s)\vert_{X\sm Y}^{-1/2},$ regarding $\underline \C^{m-1}$ as a vector subbundle of $\underline \C^m=\underline \C^{m-1}\op\underline\C.$

Thus, for any $\SU(m)$-bundle $E\ra X$ we can find a codimension 4 submanifold $Y\subset X$ and a $\SU(m)$-framing $\rho:\ul{\C}^m\vert_{X\sm Y}\,{\buildrel\cong\over\longra}\,\ab E\vert_{X\sm Y}$ of $E$ outside $Y.$

From the definition of Chern classes in terms of degeneracy cycles in Griffiths and Harris \cite[p.~412-3]{GrHa}, we see that the homology class $[Y]\in H_{n-4}(X;\Z)$ is Poincar\'e dual to the second Chern class $c_2(E)\in H^4(X;\Z).$ More generally, if $U$ is any open neighbourhood of $Y$ in $X$ then $[Y]\in H_{n-4}(U;\Z)$ is Poincar\'e dual to the compactly-supported second Chern class~$c_2(E\vert_U,\rho\vert_{U\sm Y})\in H^4_{\mathrm{cpt}}(U;\Z).$
\end{con}

\subsubsection{Two embedding theorems}

We will need the following variation on Whitney's Embedding Theorem:

\begin{thm}[Haefliger~{\cite[p.~47]{Haef}}]
\label{thm.4.5}
Let\/ $Y$ be a compact\/ $3$-manifold. Then:
\begin{itemize}
\setlength{\itemsep}{0pt}
\setlength{\parsep}{0pt}
\item[{\bf(i)}] There is an embedding\/ $Y \hookrightarrow \cS^7.$
\item[{\bf(ii)}] Any two embeddings\/ $Y\hookrightarrow \cS^7$ are isotopic through embeddings.
\end{itemize}
\end{thm}

Wall has shown the following:

\begin{thm}[Wall~{\cite[p.~567]{Wall}}]
\label{thm.4.6} 
Let\/ $Z$ be a compact connected\/ $4$-manifold with non-empty boundary. Then there exists an embedding\/ $Z\hookrightarrow\cS^7.$
\end{thm}

\subsubsection{Definition of the orientations $o^F_{Y,\rho,\io,\Psi}(E)$}

\begin{dfn}
\label{dfn.4.7}
Suppose $X$ is a compact, oriented, spin 7-manifold with flag structure $F,$ and $E \to X$ is a rank $m$ complex vector bundle with $\SU(m)$-structure. After making some arbitrary choices, we will define an orientation $o^F_{Y,\rho,\io,\Psi}(E)$ in~$\Or_E.$

As in Construction \ref{con.4.4}, choose a generic morphism $s\colon\underline\C^{m-1}\ra E,$ and from this construct a compact, oriented 3-submanifold $Y\subset X$ and a $\SU(m)$-framing $\rho:\ul{\C}^m\vert_{X\sm Y}\,{\buildrel\cong\over\longra}\,\ab E\vert_{X\sm Y}$ of $E$ outside $Y.$ By Theorem \ref{thm.4.5}(i) we may choose an embedding $\io\colon Y \hookrightarrow \cS^7.$ Set $Y'=\io(Y),$ a 3-submanifold of $\cS^7.$ Write $N_Y,N_{Y'}$ for the normal bundles of $Y,Y'$ in~$X,\cS^7.$

We claim that we may choose an isomorphism $\Psi:N_Y\ra \io^*(N_{Y'})$ of vector bundles on $Y,$ which identifies the orientations and spin structures on the total spaces of $N_Y,N_{Y'}$ induced by the orientations and spin structures on $X,\cS^7.$ Here when we say that $\Psi$ identifies the spin structures, we mean that it has a lift $\hat\Psi$ to the spin bundles of $N_Y,N_{Y'}.$

To see this, choose a spin structure on the oriented 3-manifold $Y$
and transport it along $\io$ to $Y'.$ Using the spin structures on $X,\cS^7$ we get,
by $2$-out-of-$3$ \cite[Prop.~1.15]{LaMi}, spin structures
$P_{\Spin}(N_Y) \to Y$ and $P_{\Spin}(N_{Y'}) \to Y'$ on $N_Y$ and $N_{Y'}.$
As $\dim Y = \dim Y' = 3$ and $\Spin(4)$ is $2$-connected, these are
trivial principal bundles, and therefore we may choose an oriented, spin isomorphism $\Psi$ between the normal bundles~$N_Y,N_{Y'}.$

Choose tubular neighbourhoods $U\subset X$ and $U'\subset \cS^7$ of $Y,Y'$ in $X,\cS^7,$ identified with open $\ep$-balls in $N_Y,N_{Y'}$ for small $\ep>0.$ Then $\Psi$ induces a diffeomorphism $\psi:U\ra U'$ identifying orientations and spin structures on $U,U',$ with $\psi\vert_Y=\io$ and $\d\psi\vert_{N_Y}=\Psi.$ As in Construction \ref{con.4.4}, $[Y]\in H_3(U;\Z)$ is Poincar\'e dual to~$c_2(E\vert_U,\rho\vert_{U\sm Y})\in H^4_{\mathrm{cpt}}(U;\Z).$

Define a rank $m$ complex vector bundle $E'\ra\cS^7$ with $\SU(m)$-structure by
$E'\vert_{\cS^7\sm Y'}\cong\ul{\C}^m$ and $E'\vert_{U'}\cong\psi_*(E\vert_U),$ identified over $U'\sm Y'$ by $(\psi\vert_{U\sm Y})_*(\rho).$ Write $\Xi\colon E\vert_U\ra \psi^*(E'\vert_{U'})$ for the natural isomorphism and $\rho'=\Xi\circ\rho\circ \psi^{-1}$ for the natural $\SU(m)$-framing $\ul{\C}^m\vert_{\cS^7\sm Y'}\,{\buildrel\cong\over\longra}\,\ab E'\vert_{\cS^7\sm Y'}.$ Then Theorem \ref{thm.2.15} gives a canonical excision isomorphism $\Or(\psi,\Xi,\rho,\rho')\colon \Or_E\ra \Or_{E'}$ in~\eqref{equation.2.4}. 

By Example \ref{ex.4.3}, $E'\ra\cS^7$ is stably trivial, so Definition \ref{dfn.4.1} gives an orientation $o^\mathrm{flat}(E')\in \Or_{E'}.$ As in \eqref{equation.1.9}, define $o^F_{Y,\rho,\io,\Psi}(E)\in\Or_E$ by
\begin{equation}
\label{equation.4.2}
o^F_{Y,\rho,\io,\Psi}(E)=\bigl(F\vert_U/\psi^*(F_{\cS^7}\vert_{U'})\bigr)[Y]
\cdot\Or(\psi,\Xi,\rho,\rho')^{-1}(o^\mathrm{flat}(E')),
\end{equation}
where $F_{\cS^7}$ is the unique flag structure on $\cS^7,$ and $F\vert_U/\psi^*(F_{\cS^7}\vert_{U'})$ is as in Proposition \ref{prop.3.3}(b) for the flag structures $F\vert_U$ and $\psi^*(F_{\cS^7}\vert_{U'})$ on~$U.$
\end{dfn}

\subsubsection{Uniqueness of orientations, if they exist}

Uniqueness of orientations, subject to our axioms, is explained in the outline of the proof in \S\ref{section.1}(A).

\subsection{$o^F_{Y,\rho,\io,\Psi}(E)$ is independent of choices}

We will prove the orientation $o^F_{Y,\rho,\io,\Psi}(E)$ in Definition \ref{dfn.4.7} depends only on $X,F$ and $E\ra X,$ and not on the other arbitrary choices.

\subsubsection{$o^F_{Y,\rho,\io,\Psi}(E)$ is independent of $U,U',\psi$ for fixed $Y,\rho,\io,\Psi$}
\label{subsubsection.4(B)(i)}

In the situation of Definition \ref{dfn.4.7}, let $X,F,E,Y,\rho,\io,\Psi$ be fixed, and let $U_0,U'_0,\psi_0$ and $U_1,U'_1,\psi_1$ be alternative choices for $U,U',\psi.$ Then by properties of tubular neighbourhoods we can find families $U_t,U'_t$ and $\psi_t:U_t\ra U_t'$ depending smoothly on $t\in[0,1]$ and interpolating between $U_0,U'_0,\psi_0$ and $U_1,U'_1,\psi_1.$ For each $t\in[0,1]$ we get an orientation $o^F_{Y,\rho,\io,\Psi}(E)_t$ in Definition \ref{dfn.4.7} defined using $U_t,U'_t,\psi_t.$ The families property Theorem \ref{thm.2.15}(ii) of excision isomorphisms implies that $o^F_{Y,\rho,\io,\Psi}(E)_t$ depends continuously on $t,$ and so is constant. Hence $o^F_{Y,\rho,\io,\Psi}(E)$ is independent of the choice of~$U,U',\psi.$

\subsubsection{$o^F_{Y,\rho,\io,\Psi}(E)$ is independent of $\Psi$ for fixed $Y,\rho,\io$}

We will need the following:

\begin{prop}
\label{prop.4.9}
Let\/ $Y$ be a compact\/ $n$-manifold,\/ $N \to Y$ be a rank\/ $2k$ real vector bundle with an orientation and spin structure on its fibres, and\/ $\Phi:N\ra N$ be an orientation and spin-preserving automorphism of\/ $N$ covering\/ $\id_Y:Y\ra Y.$ Suppose\/ $E \to N$ is a rank\/ $m$ complex vector bundle with\/ $\U(m)$-structure for\/ $2m\ge n+2k$ with a framing\/ $\rho$ outside a compact subset of\/ $N.$ Then there exists a\/ $\U(m)$-isomorphism\/ $\Theta\colon E \to \Phi^*(E)$ over\/ $\id_N$ with\/ $\Theta\circ\rho= \Phi^*(\rho)$ outside a compact subset of\/~$N.$

When\/ $n=3$ and\/ $2k=4,$ the same holds with\/ $\SU(m)$ in place of\/~$\U(m).$
\end{prop}

\begin{proof} By Atiyah, Bott and Shapiro \cite[Th.~12.3(ii)]{ABS}, the orientation and spin structure on $N$ determine a Thom isomorphism $\operatorname{Thom}:K^0(Y)\ra K^0_\mathrm{cpt}(N),$ a form of Bott periodicity. By naturality we have a commutative diagram
\begin{equation*}
\xymatrix@C=90pt@R=15pt{ *+[r]{K^0(Y)} \ar[d]^{\id^*_Y} \ar[r]_{\operatorname{Thom}} & *+[l]{K^0_\mathrm{cpt}(N)} \ar[d]_{\Phi^*} \\
*+[r]{K^0(Y)} \ar[r]^{\operatorname{Thom}} & *+[l]{K^0_\mathrm{cpt}(N).\!}
}\end{equation*}
Since the horizontal maps are isomorphisms we see that $\Phi^*=\id.$ Thus we have $[E,\rho]=[\Phi^*(E),\Phi^*(\rho)]$ in $K^0_\mathrm{cpt}(N).$ As we are in the stable range $2m\ge k,$ the K-theory class determines the bundle up to isomorphism, so $\Th$ exists as claimed.

For the second part, every $\Spin(4)$-bundle over a compact 3-manifold $Y$ is trivializable, and $\Omega^4 B\SU \simeq \Omega^4 B\U$ means that stably there is no difference between unitary and special unitary bundles on $N\cup \{\infty\}\cong Y^+\wedge \cS^4.$
\end{proof}

\begin{prop}
$o^F_{Y,\rho,\io,\Psi}(E)$ is independent of\/~$\Psi.$
\end{prop}

\begin{proof} In the situation of Definition \ref{dfn.4.7}, let $X,F,E,Y,\rho,\io,Y'$ be fixed, and let $\Psi_0,\Psi_1:N_Y\ra\io^*(N_{Y'})$ be alternative choices for $\Psi.$ Using the same tubular neighbourhoods $U,U'$ for $Y,Y'$ in $X,\cS^7,$ which do not affect $o^F_{Y,\rho,\io,\Psi_i}(E)$ by \S\ref{subsubsection.4(B)(i)}, these induce diffeomorphisms $\psi_0,\psi_1:U\ra U'.$ Let $E_0',E_1'\ra\cS^7$ be the corresponding $\SU(m)$-bundles, with $\SU(m)$-framings $\rho_0',\rho_1'$ over $\cS^7\sm Y'.$

Pick a spin structure on $Y\cong Y',$ which determines spin structures on the fibres $N_Y,N_{Y'}$ by 2-out-of-3 for spin structures, where $\Psi_0,\Psi_1$ preserve these spin structures. Write $\phi=\Psi_1^{-1}\ci\Psi_0:N_Y\ra N_Y,$ so that $\phi$ preserves orientations and spin structures on the fibres of $N_Y.$

Write $\cS(N_Y\oplus\underline{\R})$ for the sphere bundle of the vector bundle $N_Y\oplus\underline{\R}\ra Y,$ so that $\cS(N_Y\oplus\underline{\R})\ra Y$ is a $\cS^4$-bundle, containing $N_Y$ as an open set, and obtained by adding a point at infinity to each fibre $\R^4$ of $N_Y\ra Y,$ making the fibres $\R^4\amalg\{\iy\}=\cS^4.$ Then $\cS(N_Y\oplus\underline{\R})$ is a compact, oriented, spin 7-manifold, and $Y$ embeds in $\cS(N_Y\oplus\underline{\R})$ as the zero section of $N_Y.$ Write $\ti\phi:\cS(N_Y\oplus\underline{\R})\ra\cS(N_Y\oplus\underline{\R})$ for the diffeomorphism induced by~$\phi:N_Y\ra N_Y.$

As $U$ is a tubular neighbourhood of $Y$ in $X$ it is diffeomorphic to the bundle of open $\ep$-balls in $N_Y,$ so we can regard $U$ as an open neighbourhood of $Y$ in $N_Y$ and $\cS(N_Y\oplus\underline{\R}).$ Write $\ti E\ra \cS(N_Y\oplus\underline{\R})$ for the rank $m$ complex vector bundle with $\SU(m)$-structure given by $\ti E\vert_U\cong E\vert_U$ (identifying the open subsets $U$ in $X$ and $\cS(N_Y\oplus\underline{\R})$), and $\ti E\vert_{\cS(N_Y\oplus\underline{\R})\sm Y}\cong \ul\C^m\vert_{\cS(N_Y\oplus\underline{\R})\sm Y},$ identified over $U\sm Y$ by $\rho\vert_{U\sm Y}.$ Write $\ti\rho:\ul{\C}^m\vert_{\cS(N_Y\oplus\underline{\R})\sm Y}\,{\buildrel\cong\over\longra}\,\ab \ti E\vert_{\cS(N_Y\oplus\underline{\R})\sm Y}$ for the obvious $\SU(m)$-framing. Then $c_2(\ti E)$ is Poincar\'e dual to $[Y]\in H_3(\cS(N_Y\oplus\underline{\R});\Z).$

After stabilizing by $\ul{\C}^l$ for $l\ge 0$ with $2(m+l)\ge 7,$ using Proposition \ref{prop.4.9} on $N_Y\subset \cS(N_Y\oplus\underline{\R})$ we obtain an isomorphism of $\SU(m+l)$-bundles
\begin{equation*}
\Theta\colon \ti E\op\ul\C^l\longra \ti\phi^*(\ti E\op\ul\C^l)\cong\ti\phi^*(\ti E)\op\ul\C^l,
\end{equation*}
compatible outside a compact subset of $N_Y\subset \cS(N_Y\oplus\underline{\R})$ with the $\SU(m+l)$-framings induced by $\ti\rho.$ Thus Theorem \ref{thm.2.15} gives an isomorphism
\begin{equation}
\label{equation.4.3}
\Or(\ti\phi,\Th,\es,\es) \colon \Or_{\ti E\op\ul\C^l} \longra \Or_{\ti E\op\ul\C^l}.
\end{equation}

Let $\ti F$ be the unique flag structure on $\cS(N_Y\oplus\underline{\R})$ with $\ti F\vert_U=F\vert_U,$ regarding $U$ as an open subset of both $\cS(N_Y\oplus\underline{\R})$ and $X.$ Then combining Propositions \ref{prop.2.19} and \ref{prop.3.6} and Example \ref{ex.2.20}, we find that $\Or(\ti\phi,\Th,\es,\es)$ in \eqref{equation.4.3} is multiplication by the sign
\begin{equation}
\label{equation.4.4}
(\ti F/\ti\phi^*\ti F)[Y]=\bigl(\ti F\vert_U/\ti\phi\vert_U^*(\ti F\vert_U)\bigr)[Y]=
\bigl(F\vert_U/(\psi_1^{-1}\ci\psi_0)^*(F\vert_U)\bigr)[Y],
\end{equation}
since identifying subsets $U$ of $X$ and $\cS(N_Y\oplus\underline{\R})$ identifies $\ti\phi\vert_U$ with $\psi_1^{-1}\ci\psi_0.$

By functoriality of excision there is a commutative diagram
\begin{equation*}
\xymatrix@C=280pt@R=15pt{
*+[r]{\Or_{E\op\ul\C^l}} \ar[d]^{\Or(\psi_1^{-1}\ci\psi_0,\Theta\vert_U,\rho\op\id_{\ul\C^l},\rho\op\id_{\ul\C^l})}\ar[r]_(0.55){\Or(\id_U,\id_{E\op\ul\C^l\vert_U},\rho\op\id_{\ul\C^l},\ti\rho\op\id_{\ul\C^l})} & *+[l]{\Or_{\tilde E\op\ul\C^l}} \ar[d]_{\begin{subarray}{l}\Or(\ti\phi,\Th,\es,\es)= \\ \text{multiplication by \eqref{equation.4.4}}\end{subarray}} \\
*+[r]{\Or_{E\op\ul\C^l}} \ar[r]^(0.55){\Or(\id_U,\id_{E\op\ul\C^l\vert_U},\rho\op\id_{\ul\C^l},\ti\rho\op\id_{\ul\C^l})} & *+[l]{\Or_{\tilde E\op\ul\C^l},} }
\end{equation*}
which implies that $\Or(\psi_1^{-1}\ci\psi_0,\Theta\vert_U,\rho\op\id_{\ul\C^l},\rho\op\id_{\ul\C^l})$ is multiplication by~\eqref{equation.4.4}. 

Similarly, we have a commutative diagram
\begin{equation*}
\xymatrix@C=260pt@R=4pt{
*+[r]{\Or_{E\op\ul\C^l}} \ar[dd]^{\begin{subarray}{l} \Or(\psi_1^{-1}\ci\psi_0,\Theta\vert_U,\rho\op\id_{\ul\C^l},\rho\op\id_{\ul\C^l})= \\ \text{multiplication by \eqref{equation.4.4}}\end{subarray}}\ar@/^.3pc/[dr]^(0.7){\Or(\psi_0,\Xi_0\op\id_{\C^l\vert_U},\rho\op\id_{\ul\C^l},\rho'\op\id_{\ul\C^l})} 
\\
& *+[l]{\Or_{E'\op\ul\C^l},}
\\
*+[r]{\Or_{E\op\ul\C^l}} \ar@/_.3pc/[ur]_(0.7){\Or(\psi_1,\Xi_1\op\id_{\C^l\vert_U},\rho\op\id_{\ul\C^l},\rho'\op\id_{\ul\C^l})}  }
\end{equation*}
which implies that $\Or(\psi_i,\Xi_i\op\id_{\C^l\vert_U},\rho\op\id_{\ul\C^l},\rho'\op\id_{\ul\C^l})$ for $i=0,1$ differ by a factor \eqref{equation.4.4}. And for $i=0,1$ we have commutative diagrams
\begin{equation*}
\xymatrix@C=280pt@R=15pt{
*+[r]{\Or_E} \ar[d]^{\la_{E,\ul\C^l}\ot(-\ot_{\Z_2}o^\mathrm{flat}(\underline\C^l))}\ar[r]_{\Or(\psi_i,\Xi_i,\rho,\rho')} & *+[l]{\Or_{E'}} \ar[d]_{\la_{E',\ul\C^l}\ot(-\ot_{\Z_2}o^\mathrm{flat}(\underline\C^l))} \\
*+[r]{\Or_{E\op\ul\C^l}} \ar[r]^{\Or(\psi_i,\Xi_i\op\id_{\C^l\vert_U},\rho\op\id_{\ul\C^l},\rho'\op\id_{\ul\C^l})} & *+[l]{\Or_{E'\op\ul\C^l},} }
\end{equation*}
so that that $\Or(\psi_i,\Xi_i,\rho,\rho')$ for $i=0,1$ also differ by a factor \eqref{equation.4.4}. Hence
\begin{align*}
&o^F_{Y,\rho,\io,\Psi_0}(E)=\bigl(F\vert_U/\psi_0^*(F_{\cS^7}\vert_{U'})\bigr)[Y]
\cdot\Or(\psi_0,\Xi_0,\rho,\rho')^{-1}(o^\mathrm{flat}(E'))\\
&\quad=\bigl(F\vert_U/\psi_0^*(F_{\cS^7}\vert_{U'})\bigr)[Y]
\cdot\bigl(F\vert_U/(\psi_1^{-1}\ci\psi_0)^*(F\vert_U)\bigr)[Y]\cdot\\
&\qquad\qquad
\Or(\psi_1,\Xi_1,\rho,\rho')^{-1}(o^\mathrm{flat}(E'))\\
&\quad=\bigl(F\vert_U/\psi_1^*(F_{\cS^7}\vert_{U'})\bigr)[Y]
\cdot\Or(\psi_1,\Xi_1,\rho,\rho')^{-1}(o^\mathrm{flat}(E'))=o^F_{Y,\rho,\io,\Psi_1}(E),
\end{align*}
using \eq{equation.4.2} in the first and fourth steps, that $\Or(\psi_i,\Xi_i,\rho,\rho')$ for $i=0,1$ differ by \eqref{equation.4.4} in the second, and functoriality of $F'/F$ in Proposition \ref{prop.3.3}(b) in the third. This completes the proof.
\end{proof}

\subsubsection{$o^F_{Y,\rho,\io,\Psi}(E)$ is independent of $\io:Y\hookra\cS^7$ for fixed $Y,\rho$}
\label{subsubsection.4(B)(iii)}

In a similar way to \S\ref{subsubsection.4(B)(i)}, this is immediate from Theorem \ref{thm.4.5}(ii) and the families property Theorem \ref{thm.2.15}(ii) of excision isomorphisms.

\subsubsection{$o^F_{Y,\rho,\io,\Psi}(E)$ is independent of $s,Y,\rho$}
\label{subsubsection.4(B)(iv)}

\begin{prop}
$o^F_{Y,\rho,\io,\Psi}(E)$ is independent of\/~$s,Y,\rho.$
\end{prop}

\begin{proof} In Definition \ref{dfn.4.7}, let $s_0,s_1:\ul\C^{m-1}\ra E$ be alternative generic choices for $s,$ and let $Y_0,\rho_0,\io_0,\Psi_0,\ldots$ and $Y_1,\rho_1,\io_1,\Psi_1,\ldots$ be subsequent choices, so we have orientations $o^F_{Y_0,\rho_0,\io_0,\Psi_0}(E)$ and $o^F_{Y_1,\rho_1,\io_1,\Psi_1}(E)$ in $\Or_E.$ 

Choose a generic morphism $\check s:\underline\C^{m-1}\t[0,1] \to E\t[0,1]$ over $X\t[0,1]$ with $\check s\vert_{X\t\{i\}}=s_i$ for $i=0,1,$ and let $Z$ be the degeneracy locus of $\check s.$ Then as in Construction \ref{con.4.4}, $Z\subset X\t[0,1]$ is a compact embedded 4-submanifold with boundary~$\partial Z=(Y_0\t\{0\})\amalg(Y_1\t\{1\}).$

By genericness, $Z$ intersects the hypersurface $X\t\{t\}$ in $X\t[0,1]$ for $t\in[0,1]$ transversely, except at finitely many points $(x_i,t_i)$ for $i=1,\ldots,k,$ with $0<t_1<\cdots<t_k<1.$ Also the projection $\pi_X\vert_Z:Z\ra X$ is an immersion except at finitely many points $(\ti x_j,\ti t_j)$ for $j=1,\ldots,l,$ where~$\{t_1,\ldots,t_k\}\cap\{\ti t_1,\ldots,\ti t_l\}=\es.$

Define $Y_t=\bigl\{y\in X:(y,t)\in Z\bigr\}$ for each $t\in [0,1].$ If  
$t\in[0,1]\sm\{t_1,\ldots,t_k\}$ then $X\t\{t\}$ intersects $Z$ transversely, so $Y_t$ is a compact embedded 3-submanifold of $X,$ which depends smoothly on $t.$ But when $t=t_i,$ $Y_{t_i}$ is generally singular at $x_i,$ and the topology of $Y_t$ changes by a surgery as $t$ crosses $t_i$ in~$[0,1].$ 

For $t\in[0,1]\sm\{t_1,\ldots,t_k\}$ we have an orientation $o^F_{Y_t,\rho_t,\io_t,\Psi_t}(E)$ from Definition \ref{dfn.4.7} with $s_t=\check s\vert_{X\t\{t\}}$ and $Y_t$ in place of $s$ and $Y,$ where  \S\ref{subsubsection.4(B)(i)}--\S\ref{subsubsection.4(B)(iii)} imply these are independent of the additional choices $\io_t,\Psi_t,\ldots.$ Locally in $t$ we can make these additional choices depend smoothly on $t.$ Hence Theorem \ref{thm.2.15}(ii) implies that for $t$ in each connected component of $[0,1]\sm\{t_1,\ldots,t_k\}$ this $o^F_{Y_t,\rho_t,\io_t,\Psi_t}(E)$ depends continuously on $t,$ and hence is constant. Thus, to show that $o^F_{Y_0,\rho_0,\io_0,\Psi_0}(E)=o^F_{Y_1,\rho_1,\io_1,\Psi_1}(E),$ it suffices to prove that 
\begin{equation}
\label{equation.4.5}
o^F_{Y_{t_i-\ep},\rho_{t_i-\ep},\io_{t_i-\ep},\Psi_{t_i-\ep}}(E)=o^F_{Y_{t_i+\ep},\rho_{t_i+\ep},\io_{t_i+\ep},\Psi_{t_i+\ep}}(E)
\end{equation}
for all $i=1,\ldots,k,$ where $\ep>0$ is small.

Since $\{t_1,\ldots,t_k\}\cap\{\ti t_1,\ldots,\ti t_l\}=\es,$ if $\ep$ is small then $[t_i-\ep,t_i+\ep]$ contains no $\ti t_j$ for $j=1,\ldots,l,$ so that 
$\pi_X\vert_{\cdots}:Z\cap (X\t[t_i-\ep,t_i+\ep])\ra X$ is an immersion. As it is injective on $Y_{t_i},$ which is compact, making $\ep$ smaller we can suppose this is an embedding, so that $W_i\coloneqq\pi_X\bigl(Z\cap (X\t[t_i-\ep,t_i+\ep])\bigr)$ is an embedded 4-submanifold in $X$ with boundary $\pd W_i=Y_{t_i-\ep}\amalg Y_{t_i+\ep}.$ As the bordism $W_i$ involves only a single surgery at $(x_i,t_i),$ each connected component of $W_i$ must have nonempty boundary.

By Theorem \ref{thm.4.6} there exists an embedding $\jmath:W_i\hookra\cS^7.$ Since $X$ and $\cS^7$ are both oriented and spin, the normal bundles of $W_i$ in $X$ and in $\cS^7$ are (noncanonically) isomorphic. Hence we can choose open tubular neighbourhoods $V$ of $W_i$ in $X$ and $V'$ of $W_i'=\jmath(W_i)$ in $\cS^7$ and a spin diffeomorphism~$\chi:V\ra V'.$

Let $U_{t_i\pm\ep}$ be tubular neighbourhoods of $Y_{t_i\pm\ep}$ in $V.$ By \S\ref{subsubsection.4(B)(i)}--\S\ref{subsubsection.4(B)(iii)} we are free to define 
$o^F_{Y_{t_i\pm\ep},\rho_{t_i\pm\ep},\io_{t_i\pm\ep},\Psi_{t_i\pm\ep}}(E)$ using $\io_{t_i\pm\ep}=\chi\vert_{Y_{t_i\pm\ep}},$ $U_{t_i\pm\ep},$ $U'_{t_i\pm\ep}=\xi(U_{t_i\pm\ep})$ and $\psi_{t_i\pm\ep}=\chi\vert_{U_{t_i\pm\ep}}.$ Then we have
\begin{align*}
&o^F_{Y_{t_i-\ep},\rho_{t_i-\ep},\io_{t_i-\ep},\Psi_{t_i-\ep}}(E)\\
&=\bigl(F\vert_{U_{t_i-\ep}}/\psi_{t_i-\ep}^*(F_{\cS^7}\vert_{U'_{t_i-\ep}})\bigr)[Y]
\cdot{}\\
&\qquad \Or(\psi_{t_i-\ep},\Xi_{t_i-\ep}\vert_{U_{t_i-\ep}},\rho_{t_i-\ep},\rho'_{t_i-\ep})^{-1}(o^\mathrm{flat}(E'_{t_i-\ep}))\\
&=\bigl(F\vert_V/\chi^*(F_{\cS^7}\vert_{V'})\bigr)[Y]
\cdot\Or(\chi,\Xi_{t_i-\ep},\rho_{t_i-\ep}\vert_{X\sm W_i},\rho'_{t_i-\ep}\vert_{\cS^7\sm W_i'})^{-1}(o^\mathrm{flat}(E'_{t_i-\ep}))\\
&=\bigl(F\vert_V/\chi^*(F_{\cS^7}\vert_{V'})\bigr)[Y]
\cdot\Or(\chi,\Xi_{t_i+\ep},\rho_{t_i+\ep}\vert_{X\sm W_i},\rho'_{t_i+\ep}\vert_{\cS^7\sm W_i'})^{-1}(o^\mathrm{flat}(E'_{t_i+\ep}))\\
&=\bigl(F\vert_{U_{t_i+\ep}}/\psi_{t_i+\ep}^*(F_{\cS^7}\vert_{U'_{t_i+\ep}})\bigr)[Y]
\cdot{}\\
&\qquad \Or(\psi_{t_i+\ep},\Xi_{t_i+\ep}\vert_{U_{t_i+\ep}},\rho_{t_i+\ep},\rho'_{t_i+\ep})^{-1}(o^\mathrm{flat}(E'_{t_i+\ep}))\\
&=o^F_{Y_{t_i+\ep},\rho_{t_i+\ep},\io_{t_i+\ep},\Psi_{t_i+\ep}}(E).
\end{align*}
Here the first and fifth steps come from \eqref{equation.4.2}. In the second and fourth steps we use $\psi_{t_i\pm\ep}=\chi\vert_{U_{t_i\pm\ep}},$ expanding the open sets $U_{t_i\pm\ep},U_{t_i\pm\ep}'$ to $V,V',$ and shrinking the domains $X\sm Y_{t_i\pm\ep},\cS^7\sm Y'_{t_i\pm\ep}$ of $\rho_{t_i\pm\ep},\rho'_{t_i\pm\ep}$ to $X\sm W_i,\cS^7\sm W_i'.$ 

In the third step, with $E,V,\chi$ fixed, we deform the $\SU(m)$-framing $\rho_t\vert_{X\sm W_i}:\ul\C^m\vert_{X\sm W_i}\ra E\vert_{X\sm W_i}$ defined using $s_t$ smoothly over $t\in[t_i-\ep,t_i+\ep],$ and hence also smoothly deforming the data $E_t',\Xi_t,\rho_t'\vert_{\cS^7\sm W_i'}$ constructed using $\rho_t\vert_{X\sm W_i}.$ Theorem \ref{thm.2.15}(ii) implies that the corresponding family of orientations deforms continuously in $t\in[t_i-\ep,t_i+\ep],$ so has the same value at $t_i\pm\ep.$ This proves equation \eqref{equation.4.5}, and the proposition.
\end{proof}

\subsubsection{The orientations $o^F(E)$ are well-defined}
\label{subsubsection.4(B)(v)}

Sections \ref{subsubsection.4(B)(i)}--\ref{subsubsection.4(B)(iv)} have shown that $o^F_{Y,\rho,\io,\Psi}(E)$ in Definition \ref{dfn.4.7} depends only on $X,F,E,$ and not on the additional choices $s,Y,\rho,\io,\Psi,U,U',\psi.$ Thus we can now define canonical orientations $o^F(E)=o^F_{Y,\rho,\io,\Psi}(E)\in\Or_E$ for all $X,F$ and $\SU(m)$-bundles $E\ra X,$ as in the first part of Theorem \ref{thm.1.2}.

\subsection{Verification of the axioms}

Axiom \eqref{equation.1.5} in Theorem \ref{thm.1.2}(a) is obvious.

\begin{prop}
\label{prop.4.12}
Let\/ $E_1,E_2 \to X$ be\/ $\SU(m_1)$- and\/ $\SU(m_2)$-bundles. Then under \eqref{equation.1.2} we have\/ $o^F(E_1\oplus E_2) = o^F(E_1) \cdot o^F(E_2),$ proving Theorem\/ {\rm\ref{thm.1.2}(i)}. Taking\/ $E_2=\ul{\C}$ gives the stabilization axiom \eqref{equation.1.6} in Theorem\/~{\rm\ref{thm.1.2}(b)}.
\end{prop}

\begin{proof} In the situation of Definition \ref{dfn.4.7}, pick generic $s_k\colon \underline\C^{m_k-1}\to E_k$ for $k=1,2,$ and let $Y_k,\rho_k,\io_k,Y_k',\Psi_k,U_k,\psi_k,U_k',\ldots$ be the subsequent choices. By genericity we may assume that $Y_1 \cap Y_2=\emptyset$ and $Y'_1\cap Y'_2=\emptyset,$ and making the tubular neighbourhoods smaller we can take $U_1\cap U_2=\emptyset$ and~$U'_1\cap U'_2=\emptyset.$

As in \S\ref{subsubsection.4(B)(v)} we have $o^F(E_k)=o^F_{Y_k,\rho_k,\io_k,\Psi_k}(E_k)$ for $k=1,2.$ Also we may write $o^F(E_1\oplus E_2) =o^F_{Y,\rho,\io,\Psi}(E_1\oplus E_2),$ where $Y=Y_1\amalg Y_2,$ $\rho=\rho_1\vert_{X\sm Y}\op \rho_2\vert_{X\sm Y},$ $\io=\io_1\amalg \io_2,$ $\Psi=\Psi_1\amalg\Psi_2,$ and  $o^F_{Y,\rho,\io,\Psi}(E_1\oplus E_2)$ is defined using $U=U_1\amalg U_2,$ $\psi=\psi_1\amalg\psi_2,$ $U'=U_1'\amalg U_2',$ $E'=E_1'\op E_2',$ and~$\rho'=\rho'_1\vert_{\cS^7\sm Y'}\op \rho_2'\vert_{\cS^7\sm Y'}.$

Proposition \ref{prop.4.2}(v) gives
\begin{equation*}
o^\mathrm{flat}(E') = o^\mathrm{flat}(E_1')\cdot o^\mathrm{flat}(E_2').
\end{equation*}
By applying $\Or(\psi, \Psi,\rho,\rho')$ to this equation and using compatibility of excision with $\lambda$ and with restriction we find that
\begin{align*}
\Or&(\psi, \Psi,\rho,\rho')\bigl( o^\mathrm{flat}(E') \bigr)\\
&= \Or(\psi_1, \Psi_1,\rho_1,\rho_1') \bigl( o^\mathrm{flat}(E_1')\bigr)
\cdot\Or(\psi_2, \Psi_2,\rho_2,\rho_2') \bigl( o^\mathrm{flat}(E_2')\bigr).
\end{align*}
The proposition then follows from \eqref{equation.4.2} by multiplying this equation by
\begin{align*}
&\bigl(F\vert_U/\psi_*(F_{\cS^7}\vert_{U'})\bigr)[Y_1 \cup Y_2] \\
&\quad =\bigl(F\vert_{U_1}/(\psi_1)_*(F_{\cS^7}\vert_{U_1})\bigr)[Y_1]
\cdot\bigl(F\vert_{U_2}/(\psi_2)_*(F_{\cS^7}\vert_{U_2})\bigr)[Y_2].\qedhere
\end{align*}
\end{proof}

\begin{prop}
The excision axiom \eqref{equation.1.7} in Theorem\/~{\rm\ref{thm.1.2}(c)} holds.
\end{prop}

\begin{proof} Work in the set up of Theorem \ref{thm.1.2}(c).
Suppose that $\Phi\circ \rho^+\vert_{U^+ \setminus K^+}=\phi^*\rho^-\vert_{U^+ \setminus K^+}$ holds
for $K^+ \subset U^+$ compact.
Enlarging $K^+$ within $U^+$ to $\check K^+$ which is the closure of an open subset of $U^+,$ we can choose a smooth morphism $s^+\colon \underline\C^{m-1}\to E^+$ on $X^+$ with $\smash{s^+\vert_{X^+\sm\check K^+}=\rho^+\vert_{\underline\C^{m-1}\vert_{X^+\sm\check K^+}}},$ such that $s^+$ is generic in the interior of~$\check K^+.$ 

As in Definition \ref{dfn.4.7}, let $Y^+$ be the degeneracy locus of $s^+,$ and construct a $\SU(m)$-framing $\check\rho^+:\ul\C^m\vert_{X^+\sm Y^+}\ra E^+\vert_{X^+\sm Y^+}$ from $s^+\vert_{X^+\sm Y^+}.$ This satisfies $\check\rho^+\vert_{X^+\sm\check K^+}=\rho^+\vert_{X^+\sm\check K^+}$ as $s^+\vert_{X^+\sm\check K^+}=\rho^+\vert_{\underline\C^{m-1}\vert_{X^+\sm\check K^+}}.$ Choose an embedding $\io^+:Y^+\hookra\cS^7,$ an isomorphism of normal bundles ${\Psi^+:N_{Y^+}\ra \io^{+*}(N_{Y'})}$ for $Y'=\io^+(Y^+),$ tubular neighbourhoods $\check U^+,U'$ of $Y^+,Y'$ in $X^+,\cS^7$ with $\check U^+\subseteq U^+,$ and a spin diffeomorphism $\psi^+:\check U^+\ra U'$ with $\psi^+\vert_{Y^+}=\io^+$ and $\d\psi^+\vert_{N_{Y^+}}=\Psi^+.$ As in Definition~\ref{dfn.4.7} we get from these a vector bundle $E'\ra\cS^7$ with $\SU(m)$\nobreakdash-structure, isomorphism $\Xi^+\colon E^+\vert_{\check U^+}\ra \psi^{+*}(E'\vert_{U'})$ and $\SU(m)$-framing ${\rho':\ul{\C}^m\vert_{\cS^7\sm Y'}\,{\buildrel\cong\over\longra}\,\ab E'\vert_{\cS^7\sm Y'}}.$

Using the isomorphisms $\phi:U^+\ra U^-$ and $\Phi\colon E^+|_{U^+} \to \phi^*(E^-|_{U^-}),$ we can transport $\check K^+,Y^+,\check\rho^+,\io^+,\Psi^+,\check U^+,\psi^+,\Xi^+$ to $\check K^-,\ldots,\Xi^-$ on $X^-$ with 
\begin{gather}
\check K^-\!=\!\phi(\check K^+), \;
Y^-\!=\!\phi(Y^+), \; \check\rho^-\vert_{X^-\sm\check K^-}\!=\!\rho^-\vert_{X^-\sm\check K^-}, \; \check\rho^-\vert_{U^-\sm Y^-}\!=\!\phi_*(\check\rho^+),
\nonumber\\
\io^-=\io^+\ci\phi\vert_{Y^+}^{-1},\;\> \Psi^-=\Psi^+\ci\d\phi\vert_{N_{Y^+}}^{-1},\;\>
\check U^-=\phi(\check U^+), \;\> \psi^-=\psi^+\ci\phi\vert^{-1}_{\check U^+}.
\label{equation.4.6}
\end{gather}
Note that the data $Y',N_{Y'},U',E',\rho'$ on $\cS^7$ is the same in both $+,-$ cases. Then as in \S\ref{subsubsection.4(B)(v)} we have
\begin{equation}
\label{equation.4.7}
o^{F^\pm}(E^\pm)=o^{F^\pm}_{Y^\pm,\check\rho^\pm,\io^\pm,\Psi^\pm}(E^\pm)\in\Or_{E^\pm}.
\end{equation}

We now have
\begin{align*}
&\Or(\phi,\Phi,\rho^+,\rho^-)\bigl(o^{F^+}(E^+)\bigr)=\bigl(F^+\vert_{\check U^+}/\psi^{+*}(F_{\cS^7}\vert_{U'})\bigr)[Y^+]\cdot{} \\
&\qquad\qquad\qquad \Or(\phi,\Phi,\rho^+,\rho^-)\ci\Or(\psi^+,\Xi^+,\check\rho^+,\rho')^{-1}(o^\mathrm{flat}(E'))\\
&=\bigl(F^+\vert_{\check U^+}/\phi\vert_{\check U^+}^*(F^-\vert_{\check U^-})\bigr)[Y^+]\cdot\bigl(F^-\vert_{\check U^-}/\psi^{-*}(F_{\cS^7}\vert_{U'})\bigr)[Y^-]\cdot{}\\
&\qquad\qquad
\Or(\psi^-,\Xi^-,\check\rho^-,\rho')^{-1}(o^\mathrm{flat}(E'))\\
&=\bigl(F^+\vert_{U^+}/\phi^*(F^-\vert_{U^-})\bigr)(\al^+)\cdot o^{F^-}(E^-),
\end{align*}
using \eqref{equation.4.2} and \eqref{equation.4.7} in the first step, \eqref{equation.4.6} and functoriality of $\Or(-)$ and $F'/F$ in the second, and using \eqref{equation.4.2} and \eqref{equation.4.7} and writing $\al^+=[Y^+]$ in $H_3(U^+;\Z)$ in the third. Since $\al^+$ is Poincar\'e dual to $c_2(E^+\vert_{U^+},\rho^+)\in H^4_{\mathrm{cpt}}(U^+;\Z)$ as in Definition \ref{dfn.4.7}, this proves~\eqref{equation.1.7}.
\end{proof}

To check assertion (ii) in Theorem \ref{thm.1.2}, regarding families, let $E \to X \times P$ be a $\SU(m)$-bundle. By compactness of $X$ each $p_0 \in P$ has an open neighbourhood $P_0$ with $E|_{X\times P_0} \cong E|_{X\times \{p_0\}} \times P_0.$ By \eqref{equation.1.7} we have $o^F(E|_{X\times \{p_0\}}) \cong o^F(E|_{X\times \{p\}})$ for every $p\in P_0$ under the excision isomorphism, which depends continuously on $p.$ This completes the proof of the first part of Theorem \ref{thm.1.2}, on $\SU(m)$-bundles.

\subsection{\texorpdfstring{Extension to $\U(m)$-bundles}{Extension to U(m)-bundles}}

Finally we extend Theorem \ref{thm.1.2} to $\U(m)$-bundles. Clearly the canonical orientations $o^F(E)$ for $\U(m)$-bundles $E\ra X$ are well-defined. They also satisfy Theorem \ref{thm.1.2}(a)--(c) and (ii), since mapping the $\U(m)$-bundle $E$ to the $\SU(m+1)$-bundle $\ti E=E\op\La^mE^*$ commutes with all the operations in (a)--(c) and~(ii).

\hbadness 10000\relax

\medskip

\noindent{\small\sc D.~Joyce, The Mathematical Institute, Radcliffe
Observatory Quarter, Woodstock Road, Oxford, OX2 6GG, U.K. E-mail: {\tt joyce@maths.ox.ac.uk}.
\smallskip

\noindent M.~Upmeier, Department of Mathematics, University of Aberdeen, Fraser Noble Building, Elphinstone Rd, Aberdeen, AB24 3UE, U.K.\\E-mail: {\tt markus.upmeier@abdn.ac.uk.}}

\end{document}